\documentclass[a4paper,oneside,11pt]{amsart}

\usepackage{geometry}
\usepackage[english]{babel}
\usepackage{amsmath}
\usepackage[utf8]{inputenc}
\usepackage[T1]{fontenc}
 
\usepackage{amsthm}
\usepackage{amssymb}
\usepackage[all]{xy}
\usepackage[,colorlinks=true,citecolor= DarkBlue,linkcolor=Black]{hyperref}
\usepackage[svgnames,table]{xcolor}
\usepackage{graphicx}
\usepackage{tikz,tkz-tab}

\theoremstyle{definition}
\newtheorem{definition}{Definition}[section]

\theoremstyle{plain}
\newtheorem{theorem}{Theorem}
\newtheorem*{theorem*}{Theorem}
\newtheorem{proposition}[definition]{Proposition}
\newtheorem*{proposition*}{Proposition}
\newtheorem{lemma}[definition]{Lemma}
\newtheorem*{lemma*}{Lemma}
\newtheorem{sublemma}[definition]{Sub-lemma}
\newtheorem*{sublemma*}{Sub-lemma}
\newtheorem{corollary}[definition]{Corollary}

\newtheorem{fact}{Fact}
\newtheorem*{fact*}{Fact}

\newtheorem*{claim*}{Claim}

\theoremstyle{remark}
\newtheorem{remark}[definition]{Remark}
\newtheorem{example}{Example}

\newcommand{\R}{\mathbf{R}}

\newcommand{\Z}{\mathbf{Z}}

\newcommand{\gl}{\mathfrak{gl}}
\renewcommand{\sl}{\mathfrak{sl}}

\renewcommand{\a}{\mathfrak{a}}
\newcommand{\g}{\mathfrak{g}}
\newcommand{\h}{\mathfrak{h}}    
\renewcommand{\k}{\mathfrak{k}}    
\newcommand{\m}{\mathfrak{m}}

\newcommand{\so}{\mathfrak{so}}
\newcommand{\p}{\mathfrak{p}}

\renewcommand{\d}{\mathrm{d}}

\DeclareMathOperator{\Diff}{Diff}

\DeclareMathOperator{\Isom}{Isom}
\DeclareMathOperator{\Ad}{Ad}  
\DeclareMathOperator{\ad}{ad}

\DeclareMathOperator{\Ker}{Ker}

\DeclareMathOperator{\Span}{Span}

\DeclareMathOperator{\Jac}{Jac}

\DeclareMathOperator{\GL}{GL}
\DeclareMathOperator{\SL}{SL}

\DeclareMathOperator{\PO}{PO}

\DeclareMathOperator{\SO}{SO}

\DeclareMathOperator{\Rk}{Rk}
\DeclareMathOperator{\Conf}{Conf}

\DeclareMathOperator{\id}{id}

\DeclareMathOperator{\Supp}{Supp}

\renewcommand{\S}{\mathbf{S}}

\newcommand{\Ein}{\mathbf{Ein}}

\renewcommand{\epsilon}{\varepsilon}
\renewcommand{\geq}{\geqslant}
\renewcommand{\leq}{\leqslant}

\renewcommand{\bar}{\overline}

\usepackage{pgf,tikz}
\usepackage{mathrsfs}
\usetikzlibrary{arrows}

\title[Conformal actions of lattices]{Conformal actions of higher rank lattices on compact pseudo-Riemannian manifolds}

\author{Vincent Pecastaing}
\thanks{Partially supported by FNR grants INTER/ANR/15/11211745 and
OPEN/16/11405402.}
\address{Department of Mathematics, University of Luxembourg,
Maison du nombre, 6 avenue de la Fonte, L-4364 Esch-sur-Alzette, Luxembourg}
\email{vincent.pecastaing@uni.lu}
\date{\today}

\begin{document}

\maketitle

\begin{abstract}
We investigate conformal actions of cocompact lattices in higher-rank simple Lie groups on compact pseudo-Riemannian manifolds. Our main result gives a general bound on the real-rank of the lattice, which was already known for the action of the full Lie group (\cite{zimmer87}). When the real-rank is maximal, we prove that the manifold is conformally flat. This indicates that a global conclusion similar to that of \cite{bader_nevo} and \cite{frances_zeghib} in the case of a Lie group action might be obtained. We also give better estimates for actions of cocompact lattices in exceptional groups. Our work is strongly inspired by the recent breakthrough of Brown, Fisher and Hurtado on Zimmer's conjecture \cite{BFH}.
\end{abstract}

\tableofcontents

\section{Introduction}

Zimmer's conjectures concern actions of lattices in higher-rank semi-simple Lie groups on differentiable manifolds. It is expected that they share common features with standard algebraic actions. The most famous problem is when the dimension of the manifold is low compared to the lattice, and important breakthroughs have recently been made in this direction (\cite{BFH}).

Originally, the conjectures of Zimmer were formulated for actions of lattices which preserve a geometric structure, such as a pseudo-Riemannian metric or a symplectic form. The general idea is that there should be a universal obstruction to the existence of a non-trivial action of a given lattice on a geometric structure of a given type, see \cite{zimmer87'}, Conjecture I for a concrete formulation. In the differentiable case, the obstruction only depends on the dimension of the manifold, but for other geometric structures, more restrictive conclusions are naturally expected.

This program is motivated by earlier results of Zimmer, notably his cocycle super-rigidity theorem, that generalizes Margulis' Super-rigidity Theorem, and deals with \textit{measure preserving} dynamics of semi-simple Lie groups and their lattices. Based on this result, he obtained very strong conclusions, for instance when a semi-simple Lie group $G$, or one of its lattices $\Gamma$, acts on a compact manifold $M$, preserving a \textit{finite-type} $H$-structure and a volume form (see \cite{zimmer86}, Theorem A and F). Section \ref{ss:isometries} gives an illustration of these results, which essentially proved Zimmer's conjectures when the acting group preserves a \textit{unimodular} rigid geometric structure.

When we consider dynamics on compact geometric structures that are not unimodular, there is a priori no natural, finite invariant measure, and the problem is notably different. Nonetheless, remarkable restrictions were observed in various non-measure preserving contexts such as \cite{zimmer87}, \cite{bader_nevo}, \cite{nevo_zimmer09}, \cite{bader_frances_melnick}. All these results deal with actions of connected semi-simple Lie groups, and to our knowledge, only few results about discrete group actions on non-unimodular geometric structures exist (see for instance Theorem 1.4 of \cite{bader_frances_melnick} for ``semi-discrete'' group acting on parabolic geometries, the non-linear version of Borel's density theorem for geometric structures in \cite{iozzi}, or \cite{zeghib_projective} for groups of projective transformations of compact pseudo-Riemannian manifolds).

The contributions of the present article are results about actions of higher-rank lattices in the continuity of these works, in the framework of conformal pseudo-Riemannian geometry, which is a typical example of non-unimodular geometric structure. 

Our main result is Theorem \ref{thm:main} below. We give an upper bound on the real-rank of the lattice, which is the same as the bound given by \cite{zimmer87} when the ambient Lie group acts. This bound is achieved when lattices in $\SO(p,q)$ act on the pseudo-Riemannian analogue of the Möbius sphere. We also prove that when the lattice has maximal real-rank, the metric is conformally flat, \textit{i.e.} locally equivalent to this model space. Global conclusions could follow, and it might be proved that the manifold is very close to the model space, as in the main results of \cite{bader_nevo} and \cite{frances_zeghib} which also deal with Lie groups actions.

Following a natural analogy with \cite{BFH}, we obtain in Theorem 2 stronger estimates for actions of uniform lattices in some exceptional Lie groups.

Our approach is largely inspired by the proof of Brown-Fisher-Hurtado in \cite{BFH}. However, significant simplifications appeared in this context, due to the fact that a \textit{rigid geometric structure} is preserved (see \cite{gromov88}, \cite{kobayashi}, Chapter I). The main point from their article that we use is the construction of invariant measures in some dynamical configurations, based on Ledrappier-Young's formula (see Section \ref{ss:differentiable_dynamics}).

\subsection{Actions by pseudo-Riemannian isometries}
\label{ss:isometries}

A natural situation where the original results of Zimmer apply is when $G$ or $\Gamma$ acts by isometries on a compact pseudo-Riemannian manifold of signature $(p,q)$, \textit{i.e.} by automorphisms of an $O(p,q)$-structure. 

We remind that a (smooth) pseudo-Riemannian metric $g$ on a manifold $M$ is a smooth distribution of non-degenerate quadratic forms of signature $(p,q)$ on the tangent spaces of $M$. An isometry is a diffeomorphism that preserves this field of quadratic forms. Pseudo-Riemannian metrics are always rigid at order $1$, implying that the group of isometries $\Isom(M,g)$ is a Lie transformation group. 

For isometric actions of lattices, Zimmer's result reads:

\begin{theorem*}[Consequence of Theorem F of \cite{zimmer86}]
Let $(M,g)$ be a closed pseudo-Riemannian manifold of signature $(p,q)$ and $\Gamma$ a lattice in a semi-simple Lie group $G$ with finite center and all of whose simple factors have real-rank at least $2$. Assume that $\Gamma$ acts isometrically on $(M,g)$.

Then, either $\g$ embeds into $\so(p,q)$, or the action factorizes through a compact Lie group, \textit{i.e.} the action is of the form $\Gamma \rightarrow K \hookrightarrow \Isom (M,g)$, where $K$ is a compact Lie group.
\end{theorem*}

Of course, the obstruction $\g \hookrightarrow \so(p,q)$ is stronger than a constraint formulated with the dimension of $M$. For instance, $\sl(3,\R)$ does not embed into $\so(2,n)$ for all $n \geq 1$. Thus, if $(M,g)$ is a closed pseudo-Riemannian manifold of signature $(2,n)$, then any isometric action of $\SL(3,\Z)$ on $(M,g)$ has finite image, even though $\dim M = n+2$ could be large.

\begin{example}
If $\min(p,q) \geq 2$ and $(p,q) \neq (2,2)$, then $G=O(p,q)$ and $\Gamma = G_{\Z}$ satisfy the hypothesis of this theorem and $\Gamma$ acts on the pseudo-Riemannian torus $\mathbf{T}^{p,q} = \R^{p,q}/\Z^{p+q}$, and its action is unbounded.
\end{example}

\begin{remark}
Even though $\Isom(M,g)$ is a Lie group, Margulis' super-rigidity does not imply that an isometric action $\Gamma \rightarrow \Isom(M,g)$ extends to an action of $G$. And this is wrong in general, as it can be observed in the example of $\mathbf{T}^{p,q}$.
\end{remark}

\begin{remark}
Concerning pseudo-Riemannian isometric actions of simple Lie groups, the conclusion of Zimmer's Embedding Theorem - Theorem A of \cite{zimmer86} - gives a complete obstruction: given a non-compact, simple Lie group $G$, the existence of a locally faithfull isometric action of $G$ on a compact manifold of signature $(p,q)$ is reduced to an algebraic question on representations of $\g$.
\end{remark}

\subsection{Conformal dynamics: motivations, Lie group actions}

Let $(M,g)$ be a pseudo-Riemannian manifold of signature $(p,q)$. 

\subsubsection{Definitions and standard examples}

The conformal class of $g$ is defined as $[g] = \{\varphi g, \ \varphi \in \mathcal{C}^{\infty}(M), \ \varphi > 0\}$, and a diffeomorphism $f$ of $M$ is said to be conformal with respect to $g$ if it preserves $[g]$ setwise. An important property is that when $\dim M \geq 3$, a conformal class $[g]$ defines a rigid geometric structure on $M$. This can be interpreted by the fact that the associated $(\R_{>0} \times O(p,q))$-structure is of finite type in the sense of Cartan (\cite{kobayashi}, Chapter I), see also \cite{gromov88}. As a consequence, the group of conformal diffeomorphisms $\Conf(M,g)$ has a natural Lie group structure.

An important example of compact pseudo-Riemannian manifold is the conformal compactification of the flat pseudo-Euclidean space $\R^{p,q}$, the (pseudo-Riemannian) \textit{Einstein universe} $\Ein^{p,q}$. It is a parabolic space $\PO(p+1,q+1)/P$, where $P$ is a maximal parabolic subgroup, isomorphic to the stabilizer of an isotropic line in $\R^{p+1,q+1}$. Otherwise stated, $\Ein^{p,q}$ is the projectivized nullcone of $\R^{p+1,q+1}$, and it inherits from it a natural conformal class of signature $(p,q)$ such that $\Conf(\Ein^{p,q}) = \PO(p+1,q+1)$. When $p=0$, it is nothing else than the sphere with its standard conformal structure.

\subsubsection{Additional motivation: a generalization of Lichnerowicz conjecture}

The interest in conformal dynamics of semi-simple Lie groups and their lattices is moreover motivated by an older problem originally asked by Lichnerowicz.

It was settled in the case of Riemannian conformal geometry. Ferrand and Obata solved it (\cite{ferrand96}, \cite{obata71}) and in the compact case, their result asserts that given a compact Riemannian manifold $(M,g)$, its conformal group $\Conf(M,g)$ is non-compact if and only if $(M^n,g)$ is conformally equivalent to the round sphere $\S^n$.

For other signatures, the situation is more complicated. The natural conjecture that arose from the theorem of Ferrand and Obata was that pseudo-Riemannian manifolds with an \textit{essential conformal group} shall be classifiable, see \cite{dambra_gromov} Section 7.6. We remind that a subgroup $H < \Conf(M,g)$ is said to be essential if $H \nsubseteqq \Isom(M,g')$ for all metrics $g'$ in the conformal class of $g$. It turned out that there are many essential pseudo-Riemannian manifolds, and that obtaining a classification seems not plausible, see \cite{frances_ferrand_obata_lorentz}, \cite{frances_counter_examples}. 

\subsubsection{Anterior results for actions of connected semi-simple groups}

It is then natural to consider manifolds admitting an essential conformal group with a ``rich'' algebraic structure. The following result gives an interesting positive answer, when ``rich'' is interpreted as ``containing a semi-simple Lie subgroup of maximal real-rank''.

\begin{theorem*}[\cite{zimmer87},\cite{bader_nevo},\cite{frances_zeghib}]
Let $(M^n,g)$ be a closed pseudo-Riemannian manifold of signature $(p,q)$, with $n \geq 3$ and $p \leq q$, and let $G$ be a non-compact simple Lie group. Assume that we are given a locally faithful, conformal action $G \rightarrow \Conf(M,g)$. Then,
\begin{itemize}
\item $\Rk_{\R} G \leq p + 1$ (follows from Theorem 1 of \cite{zimmer87})
\item and if $\Rk_{\R} G = p+1$, then $\g = \so(p+1,k)$ with $p+1 \leq k \leq q+1$ and $(M,g)$ is conformally diffeomorphic to a quotient $\Gamma \setminus \tilde{\Ein^{p,q}}$, where $\Gamma$ is a discrete group acting freely, properly and conformally (\cite{bader_nevo}, Theorem 2 and \cite{frances_zeghib}).
\end{itemize}
\end{theorem*}

We recently obtained results about conformal actions of semi-simple Lie groups whose real-rank is not maximal \cite{pecastaing_smooth}, \cite{pecastaing_rang1}.

\subsection{Main result: conformal actions of uniform lattices}

Our main result gives a similar statement for conformal actions of cocompact lattices of $G$.

\begin{theorem}
\label{thm:main}
Let $(M^n,g)$ be a closed pseudo-Riemannian manifold of signature $(p,q)$, with $n \geq 3$, and $\Gamma < G$ a uniform lattice in a non-compact simple Lie group of real-rank at least $2$ and finite center. Assume that we are given $\alpha : \Gamma \rightarrow \Conf(M,g)$, a conformal action such that $\alpha(\Gamma)$ is unbounded in $\Conf(M,g)$. Then, 
\begin{itemize}
\item $\Rk_{\R} G \leq \min(p,q)+1$,
\item and when $\Rk_{\R} G = \min(p,q)+1$, $(M,g)$ is conformally flat.
\end{itemize}
\end{theorem}

We remind that a pseudo-Riemannian metric is said to be conformally flat if near every point, the metric reads $\varphi(x)(-\d x_1^2 - \cdots - \d x_p^2 + \d x_{p+1}^2 + \cdots + \d x_n^2)$ in local coordinates, where $\varphi > 0$.

Even though our conclusion is not as sharp as in the case of an action of a semi-simple Lie group, we suspect that nothing notably different may happen. The remaining problem is to consider the action of such a lattice on compact manifolds endowed with a $(\Conf(\Ein^{p,q}),\Ein^{p,q})$-structure, and we expect that these structures should be complete. We leave this problem for further investigations.

\subsection{Better bounds on the optimal index for exceptional groups}

Let $\Gamma$ be a lattice in a higher rank semi-simple Lie group $G$ with no compact factor. A famous question addressed in the Zimmer program is to determine the smallest integer $n$ such that there exist a compact manifold $M^n$ and an action $\Gamma \rightarrow \Diff(M)$ with infinite image. In the context of conformal actions of $\Gamma$, an analogous question would be to determine the ``optimal signature(s)'' for which there exists a non-trivial conformal action of $\Gamma$ on a compact manifold. A natural quantity that we want to optimize is the metric index $\min(p,q)$, which is the dimension of maximally isotropic subspaces of $g$.

\begin{definition}
We define the optimal index of $\Gamma$ as the smallest integer $k$ such that there exist a compact pseudo-Riemannian manifold $(M,g)$ of metric index $\min(p,q)=k$ and a conformal action $\alpha : \Gamma \rightarrow \Conf(M,g)$ such that $\alpha(\Gamma)$ is unbounded in $\Conf(M,g)$. We note $k_{\Gamma}$ the optimal index of $\Gamma$.
\end{definition}

The first point of Theorem \ref{thm:main} says that $k_{\Gamma} \geq \Rk_{\R}(G) - 1$ when $\Gamma$ is cocompact. Even though this bound is an equality when $\Gamma$ is a lattice in a group of the form $\SO(p,q)$, we expect that it is not the case for other groups.

An analogy can be made with the article of Brown-Fisher-Hurtado. The proof of the upper bound on $\Rk_{\R} G$ in Theorem \ref{thm:main} is based on Proposition \ref{prop:G-invariance} (contained in \cite{BFH}), which is a property of differentiable actions. In the context of Zimmer's conjecture, this property is not strong enough to obtain the bounds proved in \cite{BFH}. In fact, put together with the other techniques involved in their work, Proposition \ref{prop:G-invariance} would ``only'' imply that if a differentiable action of $\Gamma$ on a compact manifold $M$ has infinite image, then $\dim M \geq \Rk_{\R}G$ in the non-unimodular case (see \cite{Brown}, Theorem 11.1'). But this is the conjectured bound only when $G$ is locally isomorphic to $\SL(n,\R)$.

For split, simple Lie groups other than $\SL(n,\R)$, Brown-Fisher-Hurtado obtained the expected bounds by applying another property involving the resonance of the Lyapunov spectrum with the restricted root-system of $\g$, proved in \cite{BRHW}.

We could expect that this more advanced methods would imply stronger bounds in the setting of conformal dynamics. Surprisingly, it is not what happens and we get the same bound as in Theorem \ref{thm:main}, except when the restricted root system of $G$ is exceptional. For these exceptional groups, we obtain the following result.

\begin{theorem}
\label{thm:exceptional}
Let $\Gamma$ be a uniform lattice in a non-compact simple Lie group $G$ of real-rank at least $2$, with finite center, and such that the restricted root-system $\Sigma$ of $G$ is exceptional. We have the following lower bounds for its optimal index, depending on $\Sigma$:
\begin{enumerate}
\item If $\Sigma = E_6$, then $k_{\Gamma} \geq 8$.
\item If $\Sigma = E_7$, then $k_{\Gamma} \geq 14$.
\item If $\Sigma = E_8$, then $k_{\Gamma} \geq 28$.
\item If $\Sigma = F_4$, then $k_{\Gamma} \geq 7$. 
\item If $\Sigma = G_2$, then $k_{\Gamma} \geq 2$.
\end{enumerate}

In the case $\Sigma = F_4$, if $\Gamma$ has an unbounded conformal action on a compact pseudo-Riemannian manifold of signature $(7,q)$, $q \geq 7$, then this manifold is conformally flat.

In the case $\Sigma = E_8$, if $\Gamma$ has an unbounded conformal action on a compact pseudo-Riemannian manifold of signature $(28,q)$, $q \geq 28$, then this manifold is conformally flat.
\end{theorem}

\begin{remark}
\label{rem:g2}
For $\Sigma = G_2$, the inequality $k_{\Gamma} \geq 2$ is sharp and is achieved when $\g$ is the real split form of $\g_2$ (see Section \ref{ss:g2}).
\end{remark}

The specific values of these lower bounds come from the \textit{minimal resonance codimension} of $\g$, see \cite{BFH}, Definition 2.1 and Example 2.3.

Even though we get a significant improvement compared to the bound $k_{\Gamma} \geq \Rk_{\R}(G)-1$, we suspect that these lower bounds on the optimal index are still not sharp in general. The geometric conclusion that we obtain for $F_4$ and $E_8$ could possibly give a way to prove that such actions do not exist.

\subsection{Organization of the article and ideas of proofs}

As said above, our approach is inspired by that of Brown-Fisher-Hurtado in the differentiable case. Let $\Gamma < G$ be a cocompact lattice in a higher-rank simple Lie group $G$ with finite center, and let $\alpha : \Gamma \rightarrow \Conf(M,g)$ be a conformal action such that $\alpha(\Gamma)$ is unbounded.

If $\alpha : \Gamma \rightarrow \Diff(M)$ is a differentiable action, a classic construction gives an auxiliary space $M^{\alpha}$ on which $G$ acts naturally, and in which the action of $\Gamma$ is encoded. Let $A < G$ be a Cartan subspace and let $\mu$ be an $A$-invariant, $A$-ergodic measure on $M^{\alpha}$. The higher-rank version of Oseledec's Theorem yields a simultaneous Oseledec's splitting of the vertical tangent bundle $F^{\alpha}$ of $M^{\alpha}$ for all elements in $A$, and the Lyapunov exponents are linear functionals $\chi_1,\ldots,\chi_r \in \a^*$.

The point that we use from \cite{BFH} is that in general, if $r$ is small compared to a data extracted from the restricted root-system of $G$, and if $\mu$ is well chosen, then $\mu$ must be $G$-invariant. See Section \ref{ss:differentiable_dynamics} for the exact statements. In the general differentiable case, the only possible control on the maximal number of Lyapunov functionals is given by the dimension.

The starting point of our work is that when a geometric structure is preserved, we have a better control of $r$. In our case of a conformal action in signature $(p,q)$, the number of Lyapunov functionals is bounded by $2\min(p,q)+1$ and they moreover satisfy linear relations (Proposition \ref{prop:lyapunov_exponents}). We explain this in Section \ref{s:linear_relation}, after having detailed how the conformal structure of $M$ can be recovered in the vertical tangent bundle of $M^{\alpha}$ in Section \ref{s:suspension}. 

In Section \ref{s:invariant_measures}, we prove Proposition \ref{prop:no_finite_invariant_measure}, which gives an important simplification compared to the differentiable case. Its content is that if $G$ is not locally isomorphic to a subgroup of $\SO(p,q)$, then $\Gamma$ does not preserve any finite measure on $M$. This proposition is almost stated in anterior works of Zimmer and relies essentially on cocycle super-rigidity and the fact that the conformal structure is rigid, which implies that $\Gamma$ acts freely and properly on a principal bundle over $M$, the Cartan bundle.

The bound on the real-rank follows easily in Section \ref{s:bound_on_rank}. Essentially, if the rank of $G$ was larger than $\min(p,q)+1$, then the number of Lyapunov functionals would be too small compared to $\Rk_{\R}G$ and we would obtain a $G$-invariant measure by the above-mentioned argument of differentiable dynamics. This would contradict the fact that $\Gamma$ does not preserve any finite measure. We give first estimates for the proof of Theorem \ref{thm:exceptional} by applying a more advanced argument given in \cite{BRHW} (Proposition \ref{prop:resonance}). We conclude the proof of Theorem \ref{thm:exceptional} in Section \ref{s:limit_cases} by considering limit cases.

Section \ref{s:conformal_flatness} is devoted to the proof of the geometric part of Theorem \ref{thm:main}. The idea is that when $\Rk_{\R}G$ is maximal, then there still cannot exist a finite $\Gamma$-invariant measure on $M$. This forces the Lyapunov functionals of a well chosen $A$-invariant, $A$-ergodic measure on $M^{\alpha}$ to be in a special configuration, always because of the the results cited in Section \ref{ss:differentiable_dynamics} and because of the linear relations seen in Section \ref{s:linear_relation}. 

In particular, this configuration singles out a direction in $\a$ admitting a uniform vertical Lyapunov spectrum. Using local stable manifolds of the corresponding flow in $M^{\alpha}$, we interpret this fact in terms of the dynamics in $M$ of some diverging sequence $(\gamma_k)$ in $\Gamma$. We then use a property of stability of sequences of conformal maps of Frances (\cite{frances_degenerescence}) to prove that the sequence $(\gamma_k)$ has a uniform contracting behavior on an open set, and we finally derive conformal flatness of this open set by using standard arguments of conformal geometry. We conclude that the whole manifold is conformally flat by observing that any compact $\Gamma$-invariant subset of $M$ will intersect such an open set.

\subsection{Conventions and notations}

Throughout this article, unless otherwise specified, $(M^n,\bar{g})$ will always denote a smooth compact pseudo-Riemannian manifold of signature $(p,q)$, with $n = p+q \geq 3$, $G$ a non-compact, simple Lie group of real-rank at least $2$ and with finite center, and $\Gamma<G$ a uniform lattice. We will consider a conformal action of $\Gamma$ on $M$, noted $\alpha : \Gamma \rightarrow \Conf(M,\bar{g})$. We note $[\bar{g}] = \{\varphi \bar{g}, \ \varphi \in \mathcal{C}^{\infty}(M,\R_{>0})\}$ the conformal class of $\bar{g}$.

We will note $A < G$ a Cartan subspace of $G$, \textit{i.e.} a maximal closed connected abelian subgroup of $G$ such that $\Ad_{\g}(A)$ is $\R$-split. The set of restricted roots of $\ad_{\g}(\a)$ is noted $\Sigma$ and for $\lambda \in \Sigma$, we let $\g_{\lambda}$ denote the corresponding restricted root-space, and $G_{\lambda}$ the closed connected subgroup to which it is tangent.

Given a differentiable action of $G$ on a manifold $N$, we will identify an element $X \in \g$ with the vector field on $N$ defined by $X(x) = \frac{\d}{\d t}_{t=0} e^{tX}.x$ for all $x \in N$. For convenience, we also note $V(x) = \{X(x), \ X \in V\} \subset T_xN$ for any vector subspace $V \subset \g$.

The (linear) frame bundle of a vector bundle $E \rightarrow N$ of rank $n$ is the $\GL(n,\R)$-principal bundle $L(E) = \{u : \R^n \rightarrow E_x, \ x \in N, \ u \text{ linear isomorphism}\}$. A frame field (with a given regularity) is a section $\sigma : N \rightarrow L(E)$. If $F : E \rightarrow E$ is a bundle morphism over $f : N \rightarrow N$, we note $\Jac_x^{\sigma}(F) := \sigma(f(x))^{-1} F \sigma(x) \in \GL(n,\R)$ its Jacobian matrix at $x$ with respect to a frame field $\sigma$.

\section{Vertical conformal structure on the suspension space}
\label{s:suspension}

Let $G$ be a simple Lie group, with $\Rk_{\R} G \geq 2$, $\Gamma < G$ be a lattice and $M^n$ be a compact manifold. A differentiable action $\alpha : \Gamma \rightarrow \Diff(M)$ gives rise to an action of $G$ on a fibered manifold.

\subsection{Suspension space}

\begin{definition}
The suspension space of $\alpha$ is the fibration $\pi : M^{\alpha} \rightarrow G/\Gamma$ given by $M^{\alpha} = (G \times M)/\Gamma$, where $\Gamma$ acts on the product via $\gamma. (g,x) = (g\gamma,\gamma^{-1}.x)$, and where $\pi$ is the natural projection.
\end{definition}

We note $[(g,x)]$ the equivalence class of $(g,x) \in G \times M$. The fibers of $\pi$ are diffeomorphic to $M$, and the total space $M^{\alpha}$ is compact when $\Gamma$ is uniform. The full Lie group $G$ acts locally freely on $M^{\alpha}$ via $g.[(g_0,x)] = [(gg_0,x)]$. The $G$-orbits are transverse to the fibers, and define in this way a natural horizontal distribution, and the action is fiber-preserving. Moreover, the original action of $\Gamma$ is encoded in this continuous action via return maps in the fibers.

We note $F^{\alpha} \subset TM^{\alpha}$ the sub-bundle tangent to the fibers. It is a vector bundle over $M^{\alpha}$, of rank $n=\dim M$ and on which $G$ acts linearly. The tangent distribution to the $G$-orbits - namely $\{\g(x^{\alpha}), \ x^{\alpha} \in M^{\alpha}\}$ - is a natural $G$-invariant distribution in direct sum with the vertical bundle. 

If $A < G$ is a Cartan subspace and $\Sigma \subset \a^*$ are its restricted-roots, then the splitting:
\begin{equation*}
\g(x^{\alpha}) = \g_0(x^{\alpha}) \oplus \bigoplus_{\lambda \in \Sigma} \g_{\lambda}(x^{\alpha})
\end{equation*}
diagonalizes the action of $A$ on the horizontal distribution since any $g \in A$ commutes with vector fields generated by elements of $\g_0$ and if $X \in \g_{\lambda}$ and if $g = e^{X_0}$, then $D_{x^{\alpha}}g. X_{x^{\alpha}} = e^{\lambda(X_0)}X_{gx^{\alpha}}$. In particular, for any $A$-invariant probability measure $\mu$, the Lyapunov spectrum of any $g \in A$ with respect to $\mu$ is completely known in the horizontal direction and does not depend on $\mu$. A central question is to understand its vertical part.

\subsection{The conformal class induced on the vertical distribution of $M^{\alpha}$}

Let $p,q$ be two non-negative integers with $n=p+q \geq 3$. We assume that the action of $\Gamma$ on $M$ is conformal with respect to a pseudo-Riemannian metric $g$ of signature $(p,q)$. This geometric data on $M$ gives rise to a conformal structure on the vertical bundle of $M^{\alpha}$. 

\subsubsection{Classic definitions}

Let us first remind:

\begin{definition}
Let $p : E \rightarrow N$ be a vector bundle of rank $n$ over a differentiable manifold $N$. A pseudo-Riemannian metric of signature $(p,q)$ on $E$ is a smooth assignment of quadratic forms $(g_x)_{x\in N}$ of signature $(p,q)$ on the fibers of $E$.

Two metrics $g_1$ and $g_2$ on $E$ are said to be conformal if there exists a smooth function $\phi : N \rightarrow \R_{>0}$ such that $g_2 = \phi g_1$. A conformal structure of signature $(p,q)$ on $E$ is an equivalence class of conformal pseudo-Riemannian metrics of signature $(p,q)$.
\end{definition}

\begin{remark}
Equivalently, a conformal structure on $E$ is the data of a covering $U_i$ of $N$, together with a smooth metric $g_i$ of signature $(p,q)$ on $p^{-1}(U_i)$ such that for all $i,j$ verifying $U_i \cap U_j \neq \emptyset$, there exists $f_{i,j} : U_i \cap U_j \rightarrow \R_{>0}$ such that $g_j = f_{i,j} g_i$ on $U_i \cap U_j$.
\end{remark}

Let $p : E \rightarrow N$ be a vector bundle, and let $g$ be a pseudo-Riemannian metric on $E$. A bundle morphism $F : E \rightarrow E$ over a map $f : N \rightarrow N$ is said to be conformal with respect to $g$ if $F^* g$ is conformal to $g$. The function $\phi : N \rightarrow \R_{>0}$ such that $F^* g= \phi g$ is called the \textit{conformal distortion} of $F$ with respect to $g$.

If a group $H$ acts by conformal bundle automorphisms of $E$, and if for all $h \in H$ we note $\lambda(h,.)$ the conformal distortion of $h$ with respect to $g$, then $\lambda : H \times N \rightarrow \R_{>0}$ is a cocycle over the action of $H$ on $N$.

\subsubsection{Vertical conformal class on the suspension}

Let $\pi : M^{\alpha} \rightarrow G/\Gamma$ be the suspension associated to the conformal action $\alpha$. As said previously, it is a bundle with fiber $M$. In fact, we have a family of natural parametrizations of the fibers. For any $g \in G$, let $\psi_g : M \rightarrow M^{\alpha}$ be the map $\psi_g(x) = [(g,x)]$. Of course, $\psi_g$ is a proper injective immersion of $M$ into $M^{\alpha}$ whose image is the fiber over $g\Gamma$. Because any two such parametrizations of a given fiber differ by an element of $\alpha(\Gamma) < \Conf(M,\bar{g})$, we can push-forward the conformal class $[\bar{g}]$ of $M$ onto conformal classes on every fiber of $M^{\alpha}$. The following proposition asserts that the result is a $G$-invariant smooth object.

\begin{proposition}
To the conformal class $[\bar{g}]$ on $M$ corresponds a conformal class $[\bar{g}^{\alpha}]$ on the vertical distribution $F^{\alpha} \subset TM^{\alpha}$ such that all the maps $\varphi_g$ are conformal diffeomorphisms between $M$ and the fibers of $M^{\alpha}$. The vertical differential action of $G$ preserves this conformal class.
\end{proposition}

\begin{proof}
Let $p : G \rightarrow G/\Gamma$ be the natural projection. Let $\{D_i\}$ be a collection of trivializing open sets of $G$, such that $\{p(D_i)\}$ is a covering of $G/\Gamma$ and for all $i,j$, $p(D_i) \cap p(D_j)$ is a connected subset, possibly empty. Let $\sigma_i : p(D_i) \rightarrow D_i$ be the associated section. Then, for all $i,j$ such that $p(D_i)$ and $p(D_j)$ intersect, the map $\sigma_i^{-1} \sigma_j$ defined on $p(D_i) \cap p(D_j)$ takes values in $\Gamma$. By continuity, it is constant equal to some $\gamma_{ij}$.

We fix $\bar{g}$ a metric in the conformal class of $M$ and we note $\lambda : \Gamma \times M \rightarrow \R_{>0}$ the conformal distortion of $\Gamma$ with respect to $\bar{g}$. Let $U_i = \pi^{-1}(p(D_i))$, $\psi_i : D_i \times M \rightarrow U_i$ the trivialization $(g,x) \mapsto [(g,x)]$. Then, we define a metric $\bar{g}_i$ on the vertical tangent bundle of $U_i$ by sending the obvious one on the vertical tangent bundle of $D_i \times M$ via $\psi_i$. When $U_i \cap U_j \neq \emptyset$, we define $f_{ij} : U_i \cap U_j \rightarrow \R_{>0}$ by $f_{ij}(\psi_j(g,x)) = \lambda(\gamma_{ij},x)$, for all $g \in D_j \cap D_i \gamma_{ij}$ and $x \in M$. Then, $\bar{g}_i = f_{ij} \bar{g}_j$ over $U_i \cap U_j$.

This conformal structure induces on each fiber of $M^{\alpha}$ the natural conformal class given by the parametrizations $\psi_g$. The $G$-invariance is immediate since $g_0.\psi_g=\psi_{g_0g}$ for all $g,g_0 \in G$.
\end{proof}

This geometric data on the vertical bundle $F^{\alpha}$ gives some restrictions on the vertical Lyapunov spectrum of any $A$-invariant measure, where $A<G$ is any Cartan subspace. 

\section{Linear relations between Lyapunov functionals}
\label{s:linear_relation}
We assume in this section that $E \rightarrow N$ is a vector bundle of rank $n=p+q$, over a compact manifold $N$, and that $E$ is endowed with a conformal class $[g]$ of pseudo-Riemannian metrics of signature $(p,q)$. Assume that we are given a conformal action of $A=\R^k$ on $E$. Given an $A$-invariant, $A$-ergodic measure on $N$, we establish general linear relations among the associated Lyapunov functionals given by Oseledec's Theorem, that we recall below.

\subsection{Higher rank Oseledec's Theorem}

As a consequence of the higher rank version of Oseledec's Theorem (\cite{BRH}, Theorem 2.4), we obtain here:

\begin{theorem*}
Assume that a connected abelian group $A \simeq \R^k$ acts differentiably on $N$ and that its action lifts to an action by bundle automorphisms of $E$. Let $\mu$ be an $A$-invariant, $A$-ergodic probability measure on $N$. Then, there exist:
\begin{enumerate}
\item a measurable set $\Lambda \subset N$ of $\mu$-measure $1$,
\item a finite set of linear forms $\chi_1,\ldots,\chi_r \in \a^*$,
\item and a measurable, $A$-invariant splitting $E= E_1 \oplus \cdots \oplus E_r$ defined over $\Lambda$,
\end{enumerate}
such that for any Riemannian norm $\|.\|$ on $E$ and for every $x \in \Lambda$ and every $v \in E_i(x) \setminus \{0\}$,
\begin{equation*}
\frac{1}{|X|}( \log \|e^X.v\| - \chi_i(X)) \xrightarrow[\substack{|X| \to \infty \\ X \in \a}]{} 0,
\end{equation*}
and
\begin{equation*}
\frac{1}{|X|}  (\log |\det \Jac_x(e^X)| - \sum_{1 \leq i \leq r} \chi_i(X)\dim E_i(x)  ) \xrightarrow[\substack{|X| \to \infty \\ X \in \a}]{} 0,
\end{equation*}
where $\Jac_x(e^X)$ denotes the matrix of $e^X : E(x) \rightarrow E(e^X.x)$ with respect to some bounded measurable frame field of $E$.
\end{theorem*}

\begin{remark}
By compactness of $N$, the classic integrability condition of the cocycle of the action is immediate since we assume the action of $A$ smooth. We also skip the conclusion on the angles which we will not use.
\end{remark}

\subsection{Asymptotic conformal distortion and orthogonality relations}
\label{ss:conformal_distortion}

Let $\Lambda \subset N$ be the set of full measure where the conclusions of Oseledec's Theorem are valid. Let $E(x) = E_1(x) \oplus \cdots \oplus E_r(x)$ be the corresponding $A$-invariant decomposition given for all $x \in \Lambda$, and let $\chi_1,\ldots,\chi_r \in \a^*$ be the Lyapunov functionals.

We fix a metric $g$ on $E$ in the conformal class. We note $\lambda : A \times N \rightarrow \R_{>0}$ the conformal distortion of $A$ with respect to $g$. It is a cocycle over the action of $A$ on $N$ with values in $\R_{>0}$. Applying Oseledec's Theorem to this cocycle, we obtain another linear form $\chi : \a \rightarrow \R$ such that $\mu$-almost everywhere, $\frac{1}{|X|} (\log |\lambda(e^X,x)| - \chi(X)) \rightarrow 0$ as $|X| \to \infty$ in $\a$. Reducing $\Lambda$ if necessary, we assume that this holds for all $x \in \Lambda$.

\begin{remark}
By compactness of $N$, any other metric in the conformal class $[g]$ is of the form $\varphi g$, where $\varphi : N \rightarrow \R_{>0}$ takes values in a bounded interval. Thus, the linear form $\chi$ does not depend on the choice of $g$ in the conformal class. 
\end{remark}

\begin{remark}
In fact, $n\chi/2$ coincides with $\sum_{1 \leq i \leq r} \dim E_i \chi_i$. This can be seen by considering a linear cocycle of the $A$-action, lying in $\R_{>0} \times O(p,q)$. The Jacobian determinant will then be the conformal distortion to the power $n/2=\dim E/2$.
\end{remark}

There are more linear relations between the $\chi_i$'s which are coming from orthogonality relations between the Oseledec's spaces. They will be obtained by using the following observation.

\begin{lemma}
\label{lem:relations_lyapunov}
For any $i,j$ and $x \in \Lambda$, if $\chi_i+\chi_j \neq \chi$, then $E_i(x) \perp E_j(x)$.
\end{lemma}

\begin{proof}
Let us choose an element $X \in \a$ such that $\chi_i(X) + \chi_j(X) < \chi(X)$. If $\|.\|$ denotes an arbitrary Riemannian metric on $E$, then by compactness of $N$, there is $C > 0$ such that for all $x \in N$ and $u,v \in E$, we have $|g_x(u,v)| \leq C \|u\|\|v\|$. Thus, if $x \in \Lambda$ and $u,v$ are in $E_i(x)$ and $E_j(x)$ respectively, then from
\begin{equation*}
\lambda(e^X,x) |g_x(u,v)| = |g_{e^X.x}(e^X u,e^X v)| \leq C \|e^X u\| \|e^X v\|,
\end{equation*}
we get
\begin{align*}
\chi(X) \leq \chi_i(X) + \chi_j(X),
\end{align*}
unless $g_x(u,v)=0$. By the choice of $X$, we obtain $E_i(x) \perp E_j(x)$ for all $x \in \Lambda$.
\end{proof}

\subsection{General linear relations}

We still consider a vector bundle $E \rightarrow N$ over a compact manifold $N$ endowed with a conformal structure of signature $(p,q)$ with $p \leq q$, preserved by an action of an abelian Lie group $A = \R^k$. Let $\mu$ be a finite $A$-invariant, $A$-ergodic measure on $N$, and let $\Lambda \subset N$ such that $\mu(\Lambda)=1$ and $E|_{\Lambda} = \bigoplus_{1 \leq i \leq r} E_i|_{\Lambda}$ be the associated decomposition given by Oseledec's Theorem. Since $A$ acts ergodically on $(N,\mu)$ and conformally on $E$, we can assume that for all $i$, the signature of $E_i$ is constant over $\Lambda$, as well as the orthogonality relations among the $E_i$'s.

\begin{proposition}
\label{prop:lyapunov_exponents}
Let $\chi_1,\ldots,\chi_r$ be the Lyapunov functionals of $\mu$, and let $\chi \in \a^*$ be the Lyapunov functional of the distortion cocyle. Then, $r \leq 2p+1$. Moreover, we can reorder the $\chi_i$'s such that $\mu$-almost everywhere:

\begin{enumerate}
\item If $i+j \neq r+1$, then $E_i \perp E_j$.
\item If $i \leq r/2$, the subspace $E_i \oplus E_{r+1-i}$ is non-degenerate, and $E_i$ and $E_{r+1-i}$ are maximally isotropic in it. Thus, they have the same dimension.
\item If $r$ is even, then $p=q$ and all $E_i$'s are totally isotropic.
\item If $r$ is odd, then $E_{(r+1)/2}$ is non-degenerate.
\end{enumerate}

Consequently, when $r=2s$ is even, the Lyapunov functionals satisfy the relations:
\begin{equation*}
\chi_1+\chi_r = \cdots = \chi_s + \chi_{s+1} = \chi.
\end{equation*}
And when $r=2s+1$ is odd, they satisfy the relations:
\begin{equation*}
\chi_1+\chi_r = \cdots = \chi_s+\chi_{s+2} = 2\chi_{s+1} = \chi.
\end{equation*}
\end{proposition}

\begin{remark}
It has to be noted that these linear forms generate a linear subspace of $\a^*$ of dimension at most $p+1$.
\end{remark}

\begin{proof}
We permute the indices such that there is $X \in \a$ such that $\chi_1(X) < \cdots < \chi_r(X)$.

\begin{flushleft}
\textbf{Case 1}: There exists $i$ such that $E_i$ is not totally isotropic.
\end{flushleft}

\begin{lemma}
\label{lem:Ei_nondegenerate}
The space $E_i$ is non-degenerate and orthogonal to $\bigoplus_{j \neq i} E_j$, which has signature $(p',p')$ for some $p' \leq p$, and $\bigoplus_{j < i} E_j$ and $\bigoplus_{j>i}E_j$ are totally isotropic.
\end{lemma}

\begin{proof}
By Lemma \ref{lem:relations_lyapunov}, we get $\chi = 2 \chi_i$. Thus, if $j \leq i$ and $k < i$, then we have $\chi_j(X) + \chi_k(X) < \chi(X)$, and the same lemma implies that $\bigoplus_{1 \leq j < i} E_j$ is totally isotropic and orthogonal to $E_i$. Similar arguments work of course for indices greater than $i$ and we obtain that $E_i^{\perp}$ contains $\bigoplus_{1\leq j<i} E_j \oplus \bigoplus_{i < j \leq r} E_j$. The dimensions imply equality, and finally $E_i \cap E_i^{\perp} = 0$. The other claim is immediate because $E_i^{\perp}$ is non-degenerate and if $\R^{p',q'} = V_1\oplus V_2$ with $V_1,V_2$ totally isotropic, then $p'=q'$ and $\dim V_1 = \dim V_2 = p'$.
\end{proof}

Inside $\bigoplus_{j \neq i} E_j$, the subspaces $\bigoplus_{1\leq j<i} E_j$ and $\bigoplus_{r \geq j>i} E_j$ are maximally isotropic. Thus, for all $j < i$, there exists $f(j) > i$ such that $E_j$ and $E_{f(j)}$ are not orthogonal, because if not $\bigoplus_{r \geq j>i} E_j$ would not be maximally isotropic. Moreover, the integer $f(j)$ is uniquely determined by $\chi_j(X) + \chi_{f(j)}(X) = \chi(X)$. The same relation also implies that $\{j \mapsto f(j)\}$ is strictly decreasing because $\chi_1(X) < \cdots < \chi_r(X)$.

By symmetry, $f : \{1,\ldots,i-1\} \rightarrow \{i+1,\ldots,r\}$ must be a bijection, and $r$ is odd, equal to $2i-1$, and $f(j) = r+1-j$. We obtain that $E_j$ and $E_{r+1-j}$ are not orthogonal for $j < i$, implying $\chi_j+\chi_{r+1-j} = \chi$. Consequently, always by Lemma \ref{lem:relations_lyapunov}, all other couples $E_j,E_{j'}$ are orthogonal. Indeed, if for instance $j<i$ and $i<k<r+1-j$, then $\chi_j(X) + \chi_k(X) <\chi_j(X)+\chi_{r+1-j}(X) = \chi(X)$. Thus, $\chi_j+\chi_k \neq \chi$ and we can apply Lemma \ref{lem:relations_lyapunov}.

For all $j<i$, $E_j \oplus E_{r+1-j}$ is not totally isotropic and orthogonal to the sum of all other spaces. By the same argument as in the proof of Lemma \ref{lem:Ei_nondegenerate}, it must be non-degenerate. Consequently, $E_j$ and $E_{r+1-j}$ are maximally isotropic in it, and thus have the same dimension.

\begin{flushleft}
\textbf{Case 2}: For all $i$, $E_i$ is totally isotropic.
\end{flushleft}

Since the metric is non-degenerate, for all $i$, there exists $f(i)$ such that $E_i$ and $E_{f(i)}$ are not orthogonal. Thus $f(i)$ is uniquely determined by $\chi_i + \chi_{f(i)} = \chi$, proving that $f$ is strictly decreasing. Necessarily, $f(i) = r+1-i$ and $r$ must be even (if not, $E_{(r+1)/2}$ would not be totally isotropic). Consequently, if $i+j \neq r+1$ then $E_i$ and $E_j$ are orthogonal. Therefore, $\bigoplus_{i \leq r/2} E_i$ is totally isotropic, and so is $\bigoplus_{i > r/2} E_i$. These subspaces being in direct sum, the full space must have split signature $(p,p)$. 

Similarly to the end of Case 1, we conclude that $E_i \oplus E_{r+1-i}$ is non-degenerate and $E_i$ and $E_{r+1-i}$ are maximally isotropic in it.
\end{proof}

\section{Invariant measures and cocyle super-rigidity}
\label{s:invariant_measures}

From now on, we consider the main object of this article, which is a conformal action $\alpha : \Gamma \rightarrow \Conf(M,\bar{g})$, where $\Gamma$ is a cocompact lattice in a non-compact simple Lie group $G$ with finite center and of real-rank at least $2$, and $(M,\bar{g})$ a closed pseudo-Riemannian manifold of signature $(p,q)$, with $p+q \geq 3$ and $p \leq q$. The global assumption that we make is that the image of $\alpha$ in $\Conf(M,\bar{g})$ is unbounded. We still note $\pi : M^{\alpha} \rightarrow G/\Gamma$ the suspension of this action.

\subsection{Finite $\Gamma$-invariant measures}

The aim of this section is to establish the proposition below, valid also when $\Gamma$ is non-uniform, and saying that when $G$ is large enough, there are no finite, $\Gamma$-invariant measures on $M$. For instance, $\Gamma$ will have no finite orbit on $M$ and the action will be essential. 

\begin{proposition}
\label{prop:no_finite_invariant_measure}
Let $G$ be as above and assume moreover that $\g$ cannot be embedded into $\so(p,q)$. Let $\Gamma<G$ be a lattice that acts conformally on a compact pseudo-Riemannian manifold $(M,\bar{g})$ of signature $(p,q)$, and such that the image of $\Gamma$ in $\Conf(M,\bar{g})$ is unbounded. Then, $\Gamma$ does not preserve any finite measure on $M$.
\end{proposition}

\begin{remark}
We emphasize that the ideas we use in its proof are not new, and largely inspired from former works of Zimmer. To our knowledge, even though several similar results were already established, such a statement is not explicitly written or proved in the literature. For the sake of self-completeness, we have chosen to give a complete exposition of the arguments.
\end{remark}

We first observe that under our assumptions, the image of $\Gamma$ in $\Conf(M,\bar{g})$ is closed. It is a consequence of the general following result.

\begin{lemma}
\label{lem:Gamma_closed}
Let $G'$ be a Lie group and $\rho : \Gamma \rightarrow G'$ a morphism such that $\rho(\Gamma)$ is not relatively compact in $G'$. Then, $\rho(\Gamma)$ is closed in $G'$.
\end{lemma}

\begin{proof}
Let $H$ be the closure of $\rho(\Gamma)$ in $G'$. Let $\Gamma_0 = \rho^{-1}(H_0)$ be the preimage of the identity component of $H$. Then, $\Gamma_0$ is normal in $\Gamma$ and by Margulis' Normal Subgroups Theorem, must be either finite or has finite index in $\Gamma$. We claim that $\Gamma_0$ is finite. To see it, we assume to the contrary that it has finite index in $\Gamma$. Since $\Gamma_0$ has property (T), we deduce that $H_0$ also has property (T) according to Theorem 1.3.4 of \cite{bekka_de_la_harpe_valette}. Let $R \triangleleft H_0$ be its solvable radical.

\begin{itemize}
\item Case 1: $H_0/R$ is non-compact. Composing $\rho$ with the projection, we obtain a morphism $\Gamma_0 \rightarrow H_0/R$ with dense image. As it follows from Margulis' Super-rigidity Theorem, there does not exist a morphism $f : \Gamma_0 \rightarrow S$ into a connected, non-compact, semi-simple Lie group $S$ such that $f(\Gamma_0)$ is dense in $S$ for the Lie group topology, and we obtain a contradiction. We omit details here, the idea is that the projection of $f(\Gamma_0)$ on a non-compact simple factor would still have to be dense, but at the same time a lattice by super-rigidity.

\item Case 2: $H_0/R$ is compact. In this case, $H_0$ is amenable. Since it also has (T), $H_0$ itself is compact. This contradicts the fact that $\rho(\Gamma)$ is unbounded in $G'$.
\end{itemize}

Finally, we get that $\Gamma_0$ is finite, and since it is dense in $H_0$, we conclude that $\Gamma_0 = \Ker \rho$ and $H_0=\{e\}$, \textit{i.e.} $\rho(\Gamma)=H$.
\end{proof}

Thus $\Gamma$ is closed in $\Conf(M,\bar{g})$ - without assuming that $\g$ does not embed into $\so(p,q)$. We remind that the Lie group structure of $\Conf(M,\bar{g})$ is defined by considering its action on $B$, the second prolongation of the $(\R_{>0} \times O(p,q))$-structure associated to the conformal class $[\bar{g}]$ (see \cite{kobayashi}, Ch. I, Theorem 5.1). All we need to know here is that $B$ is a principal bundle over $M$, with structure group $P:=(\R_{>0} \times O(p,q)) \ltimes \R^n$, and that the action of $\Conf(M,\bar{g})$ on $M$ lifts to an action by bundle automorphisms of $B$, which is \textit{free and proper}. The differential structure on $\Conf(M,\bar{g})$ is then obtained by identifying it with any of its orbits in $B$.

Assume now that a closed subgroup $H < \Conf(M,\bar{g})$ acts on $B$, preserving a finite measure $\mu$, which we can assume to be $H$-ergodic. Then, $H$ has to be compact. To see it, consider the natural projection $p : B \rightarrow H \! \setminus \! B$. Since $H$ is closed, its action on $B$ is proper, and the target space is Hausdorff. Therefore, $p$ must be $\mu$-essentially constant by ergodicity, meaning that $H$ has an orbit of full measure in $B$ (this argument is a basic case of Proposition 2.1.10 of \cite{zimmer_ergodic}). Since the action is free, this implies that $H$ has finite Haar measure, so $H$ is compact.

Thus, the proof of Proposition \ref{prop:no_finite_invariant_measure} will be completed with the following lemma based on cocycle super-rigidity, and inspired from the arguments of \cite{zimmer84}, page 23.

\begin{lemma}
\label{prop:gamma_invariant_measures}
If $\g$ does not embed into $\so(p,q)$ and if there exists a finite $\Gamma$-invariant measure $\mu$ on $M$, then there exists a finite $\Gamma$-invariant measure $\mu_B$ on the prolongation bundle $B$.
\end{lemma}

\begin{proof}
Considering an ergodic component, we may assume that $\mu$ is $\Gamma$-ergodic. Let us note $p : B \rightarrow M$ the projection of the bundle, whose fibers are given by the free and proper right action of $P$ on $B$. The key point is that the action of $\Gamma$ on the bundle $B$ has to preserve a measurable sub-bundle with compact fiber, and this comes from Zimmer's cocycle super-rigidity. To be precise, the claim is the following.

\begin{sublemma}
\label{sublem:compact_algebraic_hull}
The exist a compact subgroup $K \subset P$, a measurable section $\sigma_K : M \rightarrow B$, and a cocycle $c_K : \Gamma \times M \rightarrow K$ such that 
\begin{equation*}
\gamma.\sigma_K(x) = \sigma_K(\gamma.x).c_K(\gamma,x)
\end{equation*}
for all $\gamma \in \Gamma$ and for $\mu$-almost every $x \in M$.
\end{sublemma}

\begin{proof}[Proof (Sub-lemma \ref{sublem:compact_algebraic_hull})]
Let us fix a bounded measurable section $\sigma : M \rightarrow B$, and let $c : \Gamma \times M \rightarrow P$ be the associated cocycle. We use Fisher-Margulis' extension of Zimmer's cocycle super-rigidity, formulated in Theorem 1.5 of \cite{fisher_margulis}. Up to passing to a finite cover of $G$ and lifting $\Gamma$ to it, they satisfy the hypothesis of this theorem. We note $P' = (\R^* \times O(p,q)) \ltimes \R^n$ the Zariski closure of $P$ in $O(p+1,q+1)$. We note $\epsilon = (-1,-\id) \in \R^* \times O(p,q)$ the central element of $P'$ such that $P' = P \sqcup \epsilon P$.

By assumption, any morphism $\g \rightarrow \so(p,q)$ is trivial. Considering the projection to the linear part, it follows that any morphism from $\g$ to $(\R\oplus \so(p,q)) \ltimes \R^n$ is also trivial. By connectedess of $G$, every morphism from $G$ to $P'$ is also trivial. Consequently, Theorem 1.5 of \cite{fisher_margulis} gives a compact subgroup $K' < P'$ such that $c$ is cohomologous to a $K'$-valued cocycle. It means that there exists a measurable $f' : M \rightarrow P'$ such that $f'(\gamma.x)^{-1} c(\gamma,x) f'(x) \in K'$ for all $\gamma$ and for almost every $x$.

We define $f : M \rightarrow P$ by $f(x)=f'(x)$ if $f(x) \in P$ and $f(x) = \epsilon f'(x)$ if not. Then, for all $\gamma$ and for $\mu$-almost every $x$, $f(\gamma.x)^{-1}c(\gamma,x)f(x) \in P\cap (K' \cup \epsilon K') =: K$. The latter is a compact subgroup of $P$ since $K' \cup \epsilon K'$ is a compact subgroup of $P'$. The section $\sigma_K(x) = \sigma(x).f(x)$ is the announced one.
\end{proof}

The set $\Lambda \subset M$ of points of $M$ at which the conclusion of Sub-lemma \ref{sublem:compact_algebraic_hull} is valid for any $\gamma \in \Gamma$ has full measure and is $\Gamma$-invariant. The section $\sigma_K$ provides a measurable trivialization $\varphi : B \rightarrow M \times P$ through which the action of an element $\gamma$ on $p^{-1}(\Lambda)$ reads $(x,p) \mapsto (\gamma.x,c_K(\gamma,x).p)$ for all $x \in \Lambda$ and $p \in P$. Thus, $\Gamma$ preserves the Borel set $\varphi^{-1}(\Lambda \times K)$, and preserves the measure $(\varphi^{-1})_*(\mu \otimes m_K)$ on it, where $m_K$ denotes the Haar measure of $K$.
\end{proof}

\subsection{Arguments from differentiable dynamics} 
\label{ss:differentiable_dynamics}

We cite in this section general results about differentiable actions of $\Gamma$ on compact manifolds which give sufficient conditions for the existence of invariant measures. They are proved and used in \cite{BRHW} and \cite{BFH}, but do not require the manifold to be low-dimensional. 

We remind the general fact:

\begin{lemma}[\cite{nevo_zimmer}, Lem. 6.1]
If $G$ preserves a finite measure on $M^{\alpha}$, then $\Gamma$ preserves a finite measure on $M$.
\end{lemma}

Thus, in our situation, the previous section implies that when $\g$ does not embed into $\so(p,q)$, it is not possible to construct any $G$-invariant finite measure on $M^{\alpha}$.

Let $A < G$ be a Cartan subspace. The heuristic of an important step in the proof of \cite{BFH} is that if the restricted root-system of $G$ is ``large'' compared to the number of vertical Lyapunov functionals of an $A$-invariant, $A$-ergodic measure $\mu$ on $M^{\alpha}$ which projects to the Haar measure of $G/\Gamma$, then $\mu$ is invariant under a lot of restricted root-spaces $G_{\lambda}$, and incidentally $G$-invariant.

Consequently, Proposition \ref{prop:gamma_invariant_measures} forbids such a configuration and implies interesting restrictions on the Lyapunov functionals.

\subsubsection{Non-zero vertical Lyapunov exponent}

The proof of Theorem \ref{thm:main} uses the following general property of differentiable actions. It does not appear explicitly in \cite{BFH}, and is used in a simpler approach to Zimmer's conjecture for cocompact lattices of $\SL(n,\R)$. An exposition of this simpler proof can be found in \cite{Brown} and \cite{cantat}.

\begin{proposition}[\cite{BFH}]
\label{prop:G-invariance}
Let $\pi : M^{\alpha} \rightarrow G / \Gamma$ be the suspension of an action $\alpha : \Gamma \rightarrow \Diff(M)$, and let $A<G$ be a Cartan subspace. Let $\mu$ be an $A$-invariant, $A$-ergodic measure on $M^{\alpha}$ such that $\pi_* \mu$ is the Haar measure of $G/\Gamma$. If there exists a non-trivial element $g \in A$ all of whose vertical Lyapunov exponents are zero, then $\mu$ is $G$-invariant.
\end{proposition}

\begin{proof}
See \cite{Brown}, proof of Theorem 11.1 and 11.1' in Section 11, or \cite{cantat} Proposition 8.7. The assumption $\dim M < \Rk_{\R}G$ is only used to exhibit an element $g \in A$ whose vertical Lyapunov spectrum is reduced to $\{0\}$ (claim (11.1) in the proof of Theorem 11.1 of \cite{Brown}, p. 46). The above statement follows from the arguments presented after this claim.
\end{proof}

In our situation of an unbounded conformal action $\alpha : \Gamma \rightarrow \Conf(M,\bar{g})$, the combination of Proposition \ref{prop:G-invariance} and Proposition \ref{prop:no_finite_invariant_measure} immediately gives:

\begin{corollary}
\label{cor:span}
Let $\alpha : \Gamma \rightarrow \Conf(M,\bar{g})$ be an unbounded conformal action in signature $(p,q)$. Let $\mu$ be an $A$-invariant, $A$-ergodic finite measure on $M^{\alpha}$ which projects to the Haar measure of $G/\Gamma$, and let $\chi_1,\ldots,\chi_r \in \a^*$ be the vertical Lyapunov exponents of $\mu$. If $\g$ does not embed into $\so(p,q)$, then $\chi_1,\ldots,\chi_r$ linearly span $\a^*$.
\end{corollary}

\subsubsection{Resonance}

A more advanced property, proved in \cite{BRHW}, is used in \cite{BFH} to obtain $G$-invariant measures on $M^{\alpha}$. Let $A<G$ be a Cartan subspace.

\begin{definition}
Let $\mu$ be an $A$-invariant, $A$-ergodic measure on $M^{\alpha}$, with vertical Lyapunov functionals $\chi_1, \ldots, \chi_r$. A restricted root $\lambda \in \Sigma$ is said to be $\mu$-resonant if there exists a vertical Lyapunov exponent $\chi_i$ and $c >0$ such that $\lambda = c \chi_i$.
\end{definition}

\begin{proposition}[\cite{BRHW}, Prop. 5.1]
\label{prop:resonance}
Let $\mu$ be an $A$-invariant, $A$-ergodic probability measure on $M^{\alpha}$ which projects to the Haar measure of $G/\Gamma$. If $\lambda \in \Sigma$ is not $\mu$-resonant, then $\mu$ is $G_{\lambda}$-invariant.
\end{proposition}

Following \cite{BFH}, we note $r(\g) = \min\{ \dim(\g'/\p'), \ \p' \text{ proper parabolic subalgebra of } \g'\}$, where $\g'$ denotes the real split simple algebra of type $\hat{\Sigma}$, where $\hat{\Sigma} = \Sigma$ when $\Sigma$ is reduced, and $\hat{\Sigma} = B_{\ell}$ when $\Sigma = (BC)_{\ell}$. This integer is called the \textit{minimal resonant codimension} of $\g$.

\begin{corollary}[\cite{BFH}]
\label{cor:resonance}
Assume that any finite $A$-invariant, $A$-ergodic measure $\mu$ on $M^{\alpha}$ has at most $r(\g)-1$ vertical Lyapunov functionals. Then, there exists a finite $G$-invariant measure on $M^{\alpha}$.
\end{corollary}

\begin{proof}
This is proved in Section 5.5 of \cite{BFH}.
\end{proof}

\section{Bound on the real-rank and further restrictions}
\label{s:bound_on_rank}

In this section, $\Gamma$ still denotes a cocompact lattice in a non-compact simple Lie group $G$ of real-rank at least $2$ and with finite center, and $\Gamma$ is still assumed to have an unbounded conformal action $\alpha : \Gamma \rightarrow \Conf(M,\bar{g})$ on a compact pseudo-Riemannian manifold $(M,\bar{g})$ of signature $(p,q)$, with $p \leq q$. 

\subsection{Upper bound on the real-rank}
\label{ss:bound_rank}

We have all the ingredients to obtain the announced bound on the real-rank, that is $\Rk_{\R}G \leq p+1$.

Let $\pi : M^{\alpha} \rightarrow G / \Gamma$ be the suspension of the action. Let $A < G$ be a Cartan subspace, $B<G$ be a Borel subgroup containing $A$ and let $\nu$ be a $B$-invariant measure on $M^{\alpha}$, which exists by amenability of $B$. Then, $\pi_* \nu$ is a $B$-invariant measure on $G/\Gamma$, thus it must be $G$-invariant (this follows for instance from the unique ergodicity of the action of the horospherical subgroup of $B$ on $G/\Gamma$), \textit{i.e.} proportional to the Haar measure. Let now $\mu$ be any $A$-ergodic component of $\nu$. Since the action of $A$ on $G/\Gamma$ is ergodic with respect to the Haar measure, it follows that $\pi_* \mu$ is also proportional to the Haar measure.

Let $\chi_1,\ldots,\chi_r \in \a^*$ be the vertical Lyapunov exponents of $A$ with respect to $\mu$. By Proposition \ref{prop:lyapunov_exponents}, we know that they span a subspace of $\a^*$ of dimension at most $p+1$. Thus, if $\Rk_{\R} G$ was greater than $p+1$, then Corollary \ref{cor:span} would imply that $\g$ embeds into $\so(p,q)$, which is obviously false since $\Rk_{\R} \g > \Rk_{\R} \so(p,q)$.

\subsection{Optimal index for exceptional Lie groups}
\label{ss:optimal_index}

Let us observe what could be derived from Proposition \ref{prop:resonance} and Corollary \ref{cor:resonance} in our situation. Let us assume that the index $p = \min(p,q)$ is \textit{optimal} for $\Gamma$, \textit{i.e.} that for all compact pseudo-Riemannian manifolds $(N,\bar{h})$, of signature $(p',q')$ such that $\min(p',q') < p$, any conformal action $\beta : \Gamma \rightarrow \Conf(N,\bar{h})$ has bounded image.

The first consequence is that $\g$ does not embed into $\so(p,q)$ because if it did, $\Gamma$ would have an unbounded action on $\Ein^{p-1,q-1}$ whose conformal group is $\PO(p,q)$. By Proposition \ref{prop:lyapunov_exponents}, for any $A$-invariant, $A$-ergodic measure $\mu$ on $M^{\alpha}$, there are at most $2p+1$ vertical Lyapunov functionals when $p<q$, and at most $2p$ when $p=q$. Thus, from Corollary \ref{cor:resonance} and Proposition \ref{prop:no_finite_invariant_measure}, we deduce that $r(\g) \leq 2p+1$. Thus, using the explicit values of $r(\g)$ given in \cite{BFH} (Example 2.3 and Appendix A.), we derive the following lower bounds for $p$. Let $\ell = \Rk_{\R}G$. 

\begin{itemize}
\item If $\Sigma = A_{\ell}$, then $r(\g) = \ell$ and we obtain $\ell \leq 2p+1$.
\item If $\Sigma = B_{\ell}, C_{\ell}, (BC)_{\ell}$, then $r(\g) = 2\ell -1$, and we get $\ell \leq p+1$.
\item If $\Sigma = D_{\ell}$, then $r(\g) = 2\ell -2$, and we get $\ell \leq p+1$.
\item If $\Sigma = E_6$, then $r(\g) = 16$, and we get $p \geq 8$.
\item If $\Sigma = E_7$, then $r(\g) = 27$, and we get $p \geq 13$.
\item If $\Sigma = E_8$, then $r(\g) = 57$, and we get $p \geq 28$.
\item If $\Sigma = F_4$, then $r(\g) = 15$, and we get $p \geq 7$.
\item If $\Sigma = G_2$, then $r(\g) = 5$, and we get $p \geq 2$.
\end{itemize}

Therefore, in all non-exceptional cases, we obtain either the same inequality as in Section \ref{ss:bound_rank}, or a worse one in the case of $A_{\ell}$. However, this immediate consequence makes no use of the linear relations satisfied by the $\chi_i$'s. We can derive better conclusions by considering the configuration of the Lyapunov functionals in the equality case $r(\g) = 2p+1$ for each exceptional restricted root-system. Unfortunately, it does not give better conclusions for non-exceptional root-system neither. This is a bit technical and postponed in Section \ref{s:limit_cases}, where the proof of Theorem \ref{thm:exceptional} will be completed.

\section{Conformal flatness in maximal real-rank}
\label{s:conformal_flatness}

In this section, we prove the geometric part of our main theorem. We fix a signature $(p,q)$, with $p+q \geq 3$ and $p \leq q$, a non-compact simple Lie group $G$ of real-rank $p+1$ and with finite center, and a cocompact lattice $\Gamma < G$. We assume that we are given a conformal action $\alpha : \Gamma \rightarrow \Conf(M,\bar{g})$ on a compact pseudo-Riemannian manifold $(M,\bar{g})$ of signature $(p,q)$ such that $\alpha(\Gamma)$ is unbounded, and we will prove that $(M,\bar{g})$ is conformally flat.

\subsection{Organization of the proof}

The starting point is that $\g$ does not embed in $\so(p,q)$ because of the real-ranks. Thus, Corollary \ref{cor:span} implies that for any Cartan subspace $A < G$ and any finite $A$-invariant, $A$-ergodic measure $\mu$ on the suspension $M^{\alpha}$, which projects to the Haar measure of $G/\Gamma$, the vertical Lyapunov exponents $\chi_1,\ldots,\chi_r$ linearly span $\a^*$. In Section \ref{ss:uniform_spectrum}, we deduce from the linear relations satisfied by the $\chi_i$'s that there exists a unique $X \in \a$ such that $\chi_1(X)=\cdots=\chi_r(X)=-1$.

A guiding principle in conformal geometry is that when there exists a sequence of conformal maps $(f_k)$ collapsing an open set to a singular set, say a point or a segment, then we can derive interesting conclusions on the conformal curvature by using conformally invariant tensors. For instance, in Lorentzian signature, if a sequence of conformal maps contracts topologically an open set to a point, then this open set is conformally flat, see \cite{frances_degenerescence}, Théorème 1.3. However, in general signature this is not true and we need some notion of ``uniformity of contraction'' to derive conformal flatness.

Here, the existence of an $\R$-split element $X \in \g$ with a uniform Lyapunov spectrum on $M^{\alpha}$ indicates that uniform contractions might be observed in the dynamics of $\Gamma$ on $M$. Using local stable manifolds of the flow of $X$ in $M^{\alpha}$, we will obtain the following in Section \ref{ss:continuous_to_discrete}. A Riemannian norm on $M$ is fixed $\|.\|$, and balls refer to its length distance.

\begin{proposition}
\label{prop:continuous_to_discrete}
Let $X \in \a$ and $\mu$ be a finite $\phi_X^t$-invariant, $\phi_X^t$-ergodic measure on $M^{\alpha}$ admitting exactly one vertical Lyapunov exponent, which is non-zero, and let $(\lambda_k) \rightarrow 0$ be a decreasing sequence. Then, there exist $x \in M$ and $g \in G$ such that $[(g,x)] \in \Supp \mu$, a sequence $(\gamma_k)$ in $\Gamma$, an increasing sequence of positive numbers $(T_k) \rightarrow \infty$ and $r>0$ such that:
\begin{enumerate}
\item $\gamma_k B(x,r) \subset B(x,r)$ for all $k$,
\item $\gamma_k : B(x,r) \rightarrow B(x,r)$ is $\lambda_k$-Lipschitz for all $k$,
\item $\gamma_k.x \rightarrow x$,
\item For all $v \in T_xM \setminus \{0\}$, $\frac{1}{T_k} \log\|D_x \gamma_k. v\| \rightarrow -1$.
\item $\frac{1}{T_k} \log |\det \Jac_x \gamma_k| \rightarrow -n$.
\end{enumerate}
\end{proposition}

For the last point, $\Jac_x \gamma_k \in \GL(n,\R)$ is the Jacobian matrix of $D_x \gamma_k$ with respect to a given measurable bounded frame field on $B(x,r)$, \textit{i.e.} a measurable section of the frame bundle of $B(x,r)$ whose image is contained in a compact subset of the bundle. Any change of this bounded frame field will not modify point (5) in the Proposition.

The next step makes a crucial use of the rigidity of the conformal structure of $(M,\bar{g})$. Using results of Frances (\cite{frances_degenerescence}) on degeneracy of conformal maps, we will prove that the derivatives of the sequence $(\gamma_k)$ obtained in Proposition \ref{prop:continuous_to_discrete} have the same exponential growth \textit{at any point} in $B(x,r)$. It will directly follow from the proposition below, proved in Section \ref{ss:stability}.

\begin{proposition}
\label{prop:stability}
Let $x \in M$, $U$ be a connected neighborhood of $x$ and $(f_k) \in \Conf(M,\bar{g})$ be a sequence such that:
\begin{enumerate}
\item any $y \in U$ admits a neighborhood $V$ such that $f_k(\bar{V}) \rightarrow \{x\}$ for the Hausdorff topology,
\item there exists $x_0 \in U$ such that for all $v \in T_{x_0}M \setminus \{0\}$, $\frac{1}{T_k} \log \|D_{x_0}f_k v\| \rightarrow -1$ and $\frac{1}{T_k} \log |\det \Jac_{x_0} f_k| \rightarrow -n$.
\end{enumerate}
Then, for all $y \in U$ and $v \in T_yM \setminus \{0\}$, $\frac{1}{T_k} \log \|D_y f_k v \| \rightarrow -1$.
\end{proposition}

Using these uniform contractions, we will deduce the following corollary in Section \ref{ss:vanishing_weyl}.

\begin{corollary}
\label{cor:local_flatness}
Let $X \in \a$ and $\mu$ be a $\phi_X^t$-invariant, $\phi_X^t$-ergodic measure on $M^{\alpha}$ admitting exactly one vertical Lyapunov exponent, which is non-zero. Then, there exist $x \in M$ and $g \in G$ such that $[(g,x)] \in \Supp \mu$ and such that $x$ admits a conformally flat neighborhood.
\end{corollary}

Finally, we will conclude in Section \ref{ss:conclusion} that any compact, $\Gamma$-invariant subset of $M$ intersects a conformally flat open set, and conformal flatness of all of $M$ follows easily.

We remind that throughout this section, $G$ is assumed to have real-rank $p+1$, that $\Gamma<G$ is a cocompact lattice, $\alpha : \Gamma \rightarrow \Conf(M,\bar{g})$ is an unbounded conformal action on a compact pseudo-Riemannian manifold of signature $(p,q)$, with $p \leq q$. We fix $A<G$ a Cartan subspace.

\subsection{A direction with uniform vertical Lyapunov spectrum}
\label{ss:uniform_spectrum}

Let $\mu$ be a finite $A$-invariant, $A$-ergodic measure on $M^{\alpha}$ which projects to the Haar measure of $G/\Gamma$. Let $\chi_1,\ldots,\chi_r \in \a^*$ be the vertical Lyapunov functionals of $A$ with respect to $\mu$ and let $\chi \in \a^*$ be the linear form associated to the conformal distortion (see Section \ref{ss:conformal_distortion}). The aim of this section is to prove the following.

\begin{lemma}
\label{lem:unifom_element}
If $p=q$, then $r=2p$ and if $p<q$ then $r=2p+1$. Moreover, there exists a unique $X \in \a$ such that:
\begin{equation*}
\chi_1(X)=\cdots=\chi_r(X)=-1.
\end{equation*}
\end{lemma}

\begin{proof}
Since $\Rk_{\R}\g > \Rk_{\R} \so(p,q)$, $\g$ does not embed into $\so(p,q)$ and Corollary \ref{cor:span} implies that $\chi_1,\ldots,\chi_r$ linearly span $\a^*$.

Assume that the Lyapunov functionals have been indexed such that they satisfy the relations given by Proposition \ref{prop:lyapunov_exponents}. The latter ensures that $r \leq 2p+1$. For all $1 \leq i \leq r/2$, we note $P_i$ the space spanned by $\chi_i$ and $\chi_{r+1-i}$. We note that $\chi \in \bigcap_{1 \leq i \leq r/2} P_i$. 

\begin{sublemma}
$\lfloor \frac{r}{2} \rfloor = p$ and $P_1,\ldots,P_p$ are planes such that for all $i$, $P_i \nsubseteq \sum_{j\neq i} P_j$.
\end{sublemma}

\begin{proof}
For all $1 \leq i,j \leq r/2$, we have $P_j \subset P_i + \R. \chi_j$ because $\chi_{r+1-j} = \chi - \chi_j$. Thus, for all $i \leq r/2$, we get $p+1 = \dim \a^* = \dim \sum_{j \leq r/2} P_j \leq \dim P_i + (\lfloor \frac{r}{2} \rfloor - 1) \leq p + 1$. Therefore, $\dim P_i=2$, $\lfloor \frac{r}{2} \rfloor = p$, and $\chi_i \notin \sum_{j \neq i}P_j$.
\end{proof}

Thus, $r \geq 2p$. So, if $p=q$, then $r=2p$ since $r \leq \dim M$. And if $p < q$, then $r$ is odd by Proposition \ref{prop:lyapunov_exponents}, so $r=2p+1$. We note that for all $i$, $\chi_i$ and $\chi$ are linearly independent.

We identify $\a^*$ with $\a$ with some Euclidean norm so that we are now dealing with a problem in a Euclidean space. For all $i \leq p$, we note that $\chi_i-\chi/2$ and $\chi_{r+1-i}-\chi/2$ are proportional and non-zero. Let $\xi_i$ be a unit vector giving this direction and $H_i = \xi_i^{\perp}$. Because $\chi$ and $\xi_i$ span $P_i$, the $\xi_i$'s are linearly independent. 

Then, $L = \bigcap_{1 \leq i \leq p} H_i = \Span(\xi_1,\ldots,\xi_p)^{\perp}$ is a line, and $\chi \notin L^{\perp}$ because if not $L$ would be orthogonal to $\chi, \xi_1,\ldots,\xi_p$ which span $\a^*$. Thus, the unique $X \in L$ such that $\chi(X)=-2$ is the announced one.
\end{proof}

\subsection{From dynamics in $M^{\alpha}$ to dynamics in $M$}
\label{ss:continuous_to_discrete}

In this section, we assume that there exist $X \in \a$ and a $\phi_X^t$-invariant, $\phi_X^t$-ergodic measure $\mu$ on $M^{\alpha}$ whose vertical Lyapunov spectrum is $\{-1\}$, with multiplicity $n = \dim M$. We let $\Lambda \subset \Supp(\mu)$ denote the set of full measure where the conclusions of Oseledec's Theorem are valid.

\subsubsection{Local stable manifolds}

Horizontally, the Lyapunov spectrum of $e^X \in A$ is simply given by the restricted root-spaces. Let us note $\Sigma \subset \a^*$ the set of restricted roots of $\a$, and $\Sigma_X^-=\{\xi \in \Sigma \ : \ \xi(X)<0\}$. We identify any element $X_0 \in \g$ with the vector field on $M^{\alpha}$ whose flow is $x^{\alpha} \mapsto e^{tX_0}.x^{\alpha}$. Note that its projection on $G/\Gamma$ is the right-invariant vector field $X_0^R$ and for any $g \in G$, $g_* X_0(x^{\alpha}) = (\Ad(g)X_0)(x^{\alpha})$. In particular, $(e^{tX})_* X_{\xi} = e^{t \xi(X)} X_{\xi}$ for any $X_{\xi} \in \g_{\xi}$. Thus, the action on the horizontal distribution being completely known, we get that the full Lyapunov spectrum of $\phi_X^t$ is $\{-1\} \cup \{\xi(X), \ \xi \in \Sigma\}$ and the strong stable distribution in $\Lambda$ is $(F^{\alpha})_{x^{\alpha}} \oplus \sum_{\xi \in \Sigma_X^-} \g_{\xi}(x^{\alpha})$.

We fix $0 < \lambda < 1$ such that $\xi(X) < -\lambda$ for all $\xi \in \Sigma_X^-$. Then, Pesin theory - for instance Theorem 16 of \cite{fathi_herman_yoccoz} - gives us a $\phi_X^t$-invariant set of full $\mu$ measure $\Lambda' \subset \Lambda$ such that for all $x^{\alpha} \in \Lambda'$, there exists a local stable manifold $W_s^{\text{loc}}(x^{\alpha})$ near $x^{\alpha}$. This $W_s^{\text{loc}}(x^{\alpha})$ is an embedded ball containing $x^{\alpha}$ and whose tangent space at $x^{\alpha}$ is the strong stable distribution, and for all $x^{\alpha}$, there exists $C(x^{\alpha}) > 0$ such that for any $y^{\alpha},z^{\alpha} \in W_s^{\text{loc}}(x^{\alpha})$, $d(\phi_X^t(y^{\alpha}),\phi_X^t(z^{\alpha})) \leq C(x^{\alpha})  d(y^{\alpha},z^{\alpha}) e^{-\lambda t}$.

\begin{lemma}
\label{lem:stable_vertical}
For all $x^{\alpha} \in \Lambda'$, projecting on $g\Gamma$, the local stable manifold $W_s^{\text{loc}}(x^{\alpha})$ contains an open neighborhood of $x^{\alpha}$ in the fiber $\pi^{-1}(g\Gamma)$.
\end{lemma}

\begin{proof}
Consider the projection $\pi(W_s^{\text{loc}}(x^{\alpha}))$ in $G / \Gamma$. It is contained in the (future) maximal stable manifold of $g\Gamma$ for the action of $e^{tX}$ on $G/\Gamma$
\begin{equation*}
W_s(g\Gamma) = \{g'\Gamma \ : \ \bar{\text{lim}}_{t \to + \infty} \frac{1}{t} \log 	d(e^{tX}g'\Gamma,e^{tX}g\Gamma) < 0\},
\end{equation*}
because the projection $\pi : M^{\alpha} \rightarrow G/\Gamma$ is a Lipschitz map. If $G_X^-<G$ is the analytic Lie subgroup associated to $\sum_{\xi \in \Sigma_X^-} \g_{\xi}$, then $W_s(g\Gamma) = G_X^-.g\Gamma$.

Shrinking $W_s^{\text{loc}}(x^{\alpha})$ if necessary, there exists a neighborhood $V$ of $\id$ in $G$ such that $\pi(W_s^{\text{loc}}(x^{\alpha})) \subset Vg\Gamma$ and such that in $Vg\Gamma$, every leaf of the foliation defined by the local action of $G_X^-$ is a closed, connected submanifold of $Vg\Gamma$ of dimension $\dim G_X^-$. By connectedness, $\pi(W_s^{\text{loc}}(x^{\alpha}))$ is contained in a single such leaf, say $\mathcal{L}$. Shrinking $V$ and $W_s^{\text{loc}}(x^{\alpha})$ once more if necessary, $\pi^{-1}(\mathcal{L}) \subset M^{\alpha}$ is a closed submanifold, of dimension $\dim M + \dim G_X^-$, which contains $W_s^{\text{loc}}(x^{\alpha})$. The latter having the same dimension, the result follows.
\end{proof}

\subsubsection{Proof of Proposition \ref{prop:continuous_to_discrete}}

Let $(\lambda_k)$ be a decreasing sequence of positive numbers. We exhibit here a sequence $(\gamma_k)$ in $\Gamma$ as claimed in Proposition \ref{prop:continuous_to_discrete}. We still note $\Lambda' \subset \Supp(\mu)$ the set of full measure where local stable manifolds are defined. Since $\mu$-almost every point is recurrent for the flow $\phi_X^t$, we choose a recurrent point $x^{\alpha} \in \Lambda'$. The idea is to consider suitable ``pseudo-return maps'' in a trivialization tube near $x^{\alpha}$.

Shrinking $W_s^{\text{loc}}(x^{\alpha})$ if necessary, we can assume that there is an open ball $D \subset G$ such that the $D\gamma$, $\gamma \in \Gamma$ are pairwise disjoint, and letting $\bar{D}$ denote its projection in $G/\Gamma$, such that $W_s^{\text{loc}}(x^{\alpha}) \subset \pi^{-1}(\bar{D}) =: T_D$. We note $\psi_D : (g,x) \in D \times M \mapsto [(g,x)] \in T_D$, and we let $\pi_M$ denote the map $\text{pr}_2 \circ \psi_D^{-1} : T_D \rightarrow M$ and $\pi_D = \text{pr}_1 \circ \psi_D^{-1} : T_D \rightarrow D$.

We fix $\|.\|_M$ and $\|.\|_G$ Riemannian metrics on $M$ and $G$. Reducing $D$ if necessary, there exists a Riemannian metric $\|.\|$ on $M^{\alpha}$ whose restriction to $T_D$ is the push-forward by $\psi_D$ of the orthogonal sum of $\|.\|_M$ and $\|.\|_G$. All the distances and balls in $M^{\alpha}$, $M$ or $M \times G$ will implicitly refer to these metrics. These choices have of course no importance, the aim is to simplify the notations.

Let $x = \pi_M (x^{\alpha})$ and $g = \pi_D(x^{\alpha}) \in D$. By a general fact of Riemannian geometry, there exists $r_0 > 0$ such that for all $r<r_0$, the ball $B(x^{\alpha},r)$ is strongly geodesically convex, in the sense that for any two points of this ball, there exists a unique minimizing geodesic joining them and contained in the ball. Thus, if $r$ is small enough, $B(x^{\alpha},r) \subset T_D$ and $\psi_D^{-1}$ embeds isometrically $B(x^{\alpha},r)$ into $D \times M$, as metric spaces. This proves that the restrictions of $\pi_M$ and $\pi_D$ to $B(x^{\alpha},r)$ are $1$-Lipschitz. Reducing $r$ once more if necessary, by Lemma \ref{lem:stable_vertical}, we may assume that we also have $B(x^{\alpha},r) \cap \pi^{-1}(g\Gamma) \subset W_s^{\text{loc}}(x^{\alpha})$. 

We fix once and for all such an $r > 0$. By construction, $\psi_D(\{g\} \times B(x,r)) \subset B(x^{\alpha},r) \cap \pi^{-1}(g\Gamma)$, and the map $y \in B(x,r) \mapsto \psi_D(g,y) \in B(x^{\alpha},r)$ is isometric. Now we choose an increasing sequence $(T_k) \rightarrow \infty$ such that:
\begin{itemize}
\item $d(\phi_X^{T_k}(x^{\alpha}),x^{\alpha}) < \frac{r}{k+1}$ for all $k\geq 1$ ;
\item $2r C(x^{\alpha}) e^{-\lambda T_k} < \min(\lambda_k,r/2)$.
\end{itemize}

We have the following observation.

\begin{fact}
Let $T>0$ be such that $e^{TX}g\Gamma \in \bar{D} \subset G/\Gamma$, and let $\gamma \in \Gamma$ be the unique element such that $e^{TX}g \gamma^{-1} \in D$. Then, for all $y^{\alpha} \in \pi^{-1}(g\Gamma)$, we have
\begin{equation}
\label{eq:lien_actions}
\pi_M(e^{TX} y^{\alpha}) = \gamma .\pi_M(y^{\alpha}).
\end{equation}
\end{fact}

\begin{proof}
If $y=\pi_M(y^{\alpha})$, then $y^{\alpha} = \psi_D(g,y) = [(g,y)]$, and $e^{TX}y^{\alpha} = [(e^{TX}g,y)] = [(e^{TX}g\gamma^{-1},\gamma.y)]$, proving the claim.
\end{proof}

For any $k \geq 1$, $\phi_X^{T_k}(x^{\alpha}) \in B(x^{\alpha},r) \subset T_D$, thus $e^{T_kX}g\Gamma \in \bar{D}$. Let $\gamma_k$ be the unique element of $\Gamma$ such that $e^{T_kX}g \gamma_k^{-1} \in D$. Let us see that these choices are convenient.

\begin{enumerate}
\item Let $y \in B(x,r)$. Then, $y^{\alpha} := \psi_D(g,y) \in W_s^{\text{loc}}(x^{\alpha}) \cap B(x^{\alpha},r)$. By the choice of $T_k$, we get that $d(\phi_X^{T_k}(x^{\alpha}),\phi_X^{T_k}(y^{\alpha})) < r/2$. By the previous observation, $\pi_M(\phi_X^{T_k}(x^{\alpha})) = \gamma_k.x$ and $\pi_M(\phi_X^{T_k}(y^{\alpha})) = \gamma_k.y$. Thus, since $\pi_M$ is $1$-Lipschitz in restriction to $B(x^{\alpha},r)$, we obtain $d(x,\gamma_k.x) < r/(k+1)$ and $d(\gamma_k.x,\gamma_k.y) < r/2$, proving the first and the third point of Proposition \ref{prop:continuous_to_discrete}.
\item For any $y,z \in B(x,r)$, since $\psi_D(g,y),\psi_D(g,z) \in W_s^{\text{loc}}(x^{\alpha}) \cap B(x^{\alpha},r)$, the same considerations as above give $d(\gamma_k.y,\gamma_k.z) \leq d(\phi_X^{T_k} \psi_D(g,y), \phi_X^{T_k} \psi_D(g,z)) < \lambda_k d(\psi_D(g,y),\psi_D(g,z))$.  Thus, $\gamma_k$ is $\lambda_k$-Lipschitz since $d(\psi_D(g,y),\psi_D(g,z)) = d(y,z)$.

\item The third point has already been observed.

\item If we differentiate the relation (\ref{eq:lien_actions}), we obtain for any vector $v^{\alpha}$ tangent to $\pi^{-1}(g\Gamma)$ and for all $k \geq 1$, $D\pi_M \circ D \phi_X^{T_k} (v^{\alpha}) = D\gamma_k \circ D\pi_M (v^{\alpha})$. Since $D\pi_M$ preserves the Riemannian norm in the vertical direction, this proves the fourth point because $\frac{1}{T} \log \|D_{x^{\alpha}} \phi_X^T v^{\alpha}\| \rightarrow -1$ by assumption on $X$.

\item We choose a bounded measurable frame field on $TM$, that we pullback on the vertical tangent bundle of $T_D$ via $\pi_M$. We can then arbitrarily extend it into a bounded measurable frame field on $F^{\alpha}$. With respect to it, we have by construction $\Jac_{x^{\alpha}} \phi_X^{T_k} = \Jac_{x} \gamma_k$ for all $k \geq 1$, and the result follows.
\end{enumerate}

\subsection{Strong stability of sequences of conformal maps}
\label{ss:stability}

In this section, we establish Proposition \ref{prop:stability}. We start by introducing some tools of conformal geometry.

\subsubsection{Definitions and general results}

We note $P = (\R_{>0} \times O(p,q)) \ltimes \R^n$ and $\pi : B \rightarrow M$ the $P$-principal bundle obtained by the prolongation procedure (we no longer work with the suspension of the action, and forget that $\pi$ used to denote its projection). We call $\pi : B \rightarrow M$ the Cartan bundle associated to $(M,[\bar{g}])$. We remind that any conformal map $f$ of $M$ lifts to a bundle automorphism of $B$, and that the action of $\Conf(M,\bar{g})$ on $B$ is free and proper. Consequently, if a sequence $(f_k) \rightarrow \infty$ is such that $f_k(x) \rightarrow x_{\infty}$ for $x,x_{\infty} \in M$, then $f_k(b) \rightarrow \infty$ for any $b$ in the fiber of $x$. The notion of holonomy sequence quantifies the divergence of $(f_k(b))$ in the fiber direction.

\begin{definition}
Let $(f_k) \in \Conf(M,\bar{g})$ be a sequence of conformal maps and $x,y \in M$ such that $f_k(x) \rightarrow y$. A sequence $(p_k)$ in $P$ is said to be a \textit{holonomy sequence of $(f_k)$ at $x$} if there exists a bounded sequence $(b_k)$ in $\pi^{-1}(x)$ such that $f_k(b_k).p_k^{-1}$ also stays in a bounded domain of $B$.
\end{definition}

The holonomy sequence of $(f_k)$ at $x$ is uniquely defined up to compact perturbations, \textit{i.e.} for any two holonomy sequences $(p_k)$ and $(p_k')$, there exist $(l_k^1),(l_k^2)$ bounded sequences in $P$ such that $p_k' = l_k^1p_kl_k^2$ for all $k$ (we remind that the action of $P$ on $B$ is free and proper).

We will use the following important observation of Frances.

\begin{lemma}[\cite{frances_degenerescence}, Lem. 6.1]
\label{lem:frances_stability}
Assume that a sequence of conformal maps $(f_k)$ of $(M,\bar{g})$ converges for the $\mathcal{C}^0$-topology to a continuous map $f : M \rightarrow M$. Then, there exists a sequence $(p_k)$ which is a holonomy sequence of $(f_k)$ at any point of $M$.
\end{lemma}

\subsubsection{Lyapunov regularity and holonomy sequences}
Let $(T_k) \rightarrow \infty$ be a sequence of positive numbers. We remind the following definition (see for instance \cite{kaimanovich}, Definition in Section 4).

\begin{definition}
Let $(g_k)$ be a sequence of matrices in $\GL(n,\R)$. We say that $(g_k)$ is $(T_k)$-Lyapunov regular if there exist a flag $\R^n=E_r \supsetneq E_{r-1} \supsetneq \cdots \supsetneq E_1 \supsetneq E_0= \{0\}$ and numbers $\chi_1 < \cdots < \chi_r$ such that $\frac{1}{T_k} \log |g_k v| \rightarrow \chi_i$ for all $1 \leq i \leq r$ and $v \in E_i \setminus E_{i-1}$, and $\frac{1}{T_k} \log |\det g_k| \rightarrow \sum_i \chi_i (\dim E_i - \dim E_{i-1})$. We will say that $(g_k)$ is uniformly $(T_k)$-Lyapunov regular when $r=1$.

A sequence $(f_k)$ of diffeomorphisms of $M$ is said to be (resp. uniformly) $(T_k)$-Lyapunov regular at $x$ if in some (equivalently any) bounded measurable frame field, $\Jac_x f_k$ is (resp. uniformly) $(T_k)$-Lyapunov regular.
\end{definition}

The following fact is a consequence of Theorem 4.1 of \cite{kaimanovich} when $T_k = k$. The proof is easily adaptable, but we give an elementary one for this basic situation.

\begin{lemma}
\label{lem:uniform_lyap_regular}
A sequence $(g_k)$ in $\GL(n,\R)$ is uniformly $(T_k)$-Lyapunov regular if and only if the limit $\chi_{\det} := \lim \frac{1}{T_k} \log |\det g_k|$ exists and $\frac{1}{T_k} \log \|g_k\| \rightarrow \frac{\chi_{\det}}{n}$.
\end{lemma}

\begin{proof}
Replacing $g_k$ by $|\det g_k|^{-1/n}g_k$, we may assume that $|\det g_k|=1$ and the statement is $\frac{1}{T_k}\log |g_k v| \rightarrow 0$ for any $v \neq 0$ if and only if $\frac{1}{T_k} \log \|g_k\| \rightarrow 0$. 

Let us assume that $\frac{1}{T_k}\log |g_k v| \rightarrow 0$ for any $v \neq 0$. If $|v|=1$, then $|g_kv| \leq \|g_k\|$ implies $0 \leq \underline{\lim} \frac{1}{T_k} \log \|g_k\|$. If $(v_1,\ldots,v_n)$ is an orthonomal basis, then $\|g_k\| \leq n \max |g_k v_i|$, implying $\overline{\lim} \frac{1}{T_k} \log \|g_k\| \leq 0$, and then $\frac{1}{T_k} \log \|g_k\| \rightarrow 0$. 

Let us assume now $\frac{1}{T_k} \log \|g_k\| \rightarrow 0$. We get that if $v \neq 0$, then $\overline{\lim} \frac{1}{T_k} \log |g_k v| \leq 0$. Since $|\det g_k| = 1$, we have $\frac{1}{\|g_k\|} \leq \|g_k^{-1}\| \leq \|g_k\|^{n-1}$ and it follows that $\frac{1}{T_k} \log \|g_k^{-1}\| \rightarrow 0$. Therefore, writing $|v| \leq \|g_k^{-1}\| |g_k v|$, we deduce $0 \leq \underline{\lim}\frac{1}{T_k} \log |g_k v|$.
\end{proof}

Consequently, if $(l_k),(l_k')$ are bounded sequences in $\GL(n,\R)$, then $(g_k)$ is uniformly $(T_k)$-Lyapunov regular if and only if $(l_kg_kl_k')$-is uniformly $(T_k)$-Lyapunov regular, and they have the same Lyapunov exponent. Indeed, if $h_k=l_kg_kl_k'$, then $\frac{1}{T_k} \log |\det h_k| \rightarrow \chi_{\det}$ and from $\|h_k\| \leq \|l_k\| \|g_k\| \|l_k'\|$ and $\|g_k\| \leq \|l_k^{-1}\| \|h_k\| \|l_k'^{-1}\|$ we get
\begin{equation*}
\overline{\lim} \frac{1}{T_k} \log \|h_k\| \leq \frac{\chi_{\det}}{n} \leq \underline{\lim} \frac{1}{T_k} \log \|h_k\|.
\end{equation*}

We note $\rho : P \rightarrow \GL(\R^n)$ the linear representation given by the projection on the first factor of $P = (\R_{>0} \times O(p,q)) \ltimes \R^n$. We prove now:

\begin{lemma}
\label{lem:lyapunov_holonomy}
Let $(f_k)$ be a sequence of conformal maps of $(M,\bar{g})$ and $x \in M$ such that $(f_k(x)) \rightarrow x_{\infty}$. The following are equivalent.
\begin{enumerate}
\item $(f_k)$ is Lyapunov regular at $x$, with Lyapunov exponents $\chi_i$ of multiplicity $d_i$.
\item For any $b$ in the fiber of $x$ and any sequence $(p_k)$ in $P$ such that $f_k(b).p_k^{-1} \rightarrow b_{\infty}$, for some $b_{\infty}$ in the fiber of $x_{\infty}$, the sequence $\rho(p_k)$ is Lyapunov regular with Lyapunov exponents $\chi_i$ and multiplicity $d_i$.
\end{enumerate}
\end{lemma}

\begin{proof}
We note $G' = \PO(p+1,q+1)$ and identify $P$ with the parabolic subgroup of $G'$ fixing a given isotropic line in $\R^{p+1,q+1}$. Then, the representation $\rho$ is the representation $P \rightarrow \GL(\g'/\p)$ induced by the adjoint representation of $G'$. 

We make use of the Cartan connection $\omega \in \Omega^1(B,\g')$ defined by the conformal structure of $M$. For any $b \in B$, following \cite{sharpe}, Chap. 5, Theorem 3.15., we let $\psi_{b} : T_xM \rightarrow \g'/\p$ denote the linear isomorphism defined by $\psi_{b}(v) = \omega_{b}(\hat{v}) \text{ mod.} \p$, for any $\hat{v} \in T_b B$ projecting to $v$. It satisfies the equivariance property $\psi_{b.p} = \rho(p^{-1}) \psi_b$ for any $b \in B$ and $p \in P$. Moreover, for any conformal map $f$, we have $\psi_{f(b)} \circ D_x f = \psi_b$. 

Therefore, if $(p_k)$, $b$, $b_{\infty}$ are as in (2), and if $x_k := f_k(x)$, $b_k := f_k(b)p_k^{-1}$, $x_0=x$ and $b_0=b$, then we have for $k\geq 1$
\begin{equation*}
\psi_{b_k} \circ D_xf_k = \rho(p_k) \psi_b.
\end{equation*}
The map $b \mapsto \psi_b$ from $B$ to the frame bundle of $M$ being continuous, the sequence $\psi_{b_k}$ is a bounded sequence of linear frames. Thus, if $\sigma : M \rightarrow \mathcal{F}(M)$ is a bounded measurable frame field, then there is a bounded sequence $(l_k)$ in $\GL(\g'/\p)$ such that $\sigma(x_k) = l_k.\psi_{b_k}$, and we get that $\Jac_x^{\sigma} (f_k) = l_k . \rho(p_k).l_0^{-1}$. Thus, $\Jac_x^{\sigma}(f_k)$ is Lyapunov regular if and only if $\rho(p_k)$ is Lyapunov regular, with the same exponents and multiplicities.
\end{proof}

\subsubsection{Proof of Proposition \ref{prop:stability}}

Let $(f_k)$ be a sequence in $\Conf(M,\bar{g})$, $x \in M$, a neighborhood $U$ of $x$ such that any $y \in U$ admits a neighborhood $V$ such that $f_k(\bar{V}) \rightarrow \{x\}$, and $x_0 \in U$ such that $\frac{1}{T_k} \log \|D_{x_0}f_k v\| \rightarrow -1$ for all $v \in T_{x_0}M \setminus \{0\}$ and $\frac{1}{T_k} \log |\det \Jac_{x_0} f_k| \rightarrow -n$, \textit{i.e.} $(f_k)$ is uniformly $(T_k)$-Lyapunov regular at $x_0$ with exponent $-1$.

Since $(f_k|_U)$ converges for the $\mathcal{C}^0$-topology to the constant map equal to $x$, Lemma \ref{lem:frances_stability} implies that there exists a sequence $(p_k)$ in $P$ which is a holonomy sequence for any point in $U$. If $b$ is fixed in the fiber of $x$, then for any point $y \in U$, there exist bounded sequences $(l_k),(l_k')$ in $P$ and a point $b'$ in the fiber of $y$ such that $f_k(b').(l_kp_kl_k')^{-1} \rightarrow b$.

Applying Lemma \ref{lem:lyapunov_holonomy} at $x_0$ we obtain that there exist $l_k,l_k'$ such that $\rho(l_kp_kl_k')$ is uniformly $(T_k)$-Lyapunov regular with exponent $-1$. Thus, Lemma \ref{lem:uniform_lyap_regular} implies that $\rho(p_k)$ has the same property. Now, if $y \in U$, and if $m_k,m_k'$ are relatively compact and such that $f_k(b').(m_kp_km_k')^{-1} \rightarrow b$, then $\rho(m_kp_km_k')$ is also uniformly $(T_k)$-Lyapunov regular with exponent $-1$ and Lemma \ref{lem:lyapunov_holonomy} implies that $(f_k)$ is uniformly $(T_k)$-Lyapunov regular at $y$, with exponent $-1$, completing the proof of Proposition \ref{prop:stability}.

\subsection{Local vanishing of the conformal curvature: proof of Corollary \ref{cor:local_flatness}}
\label{ss:vanishing_weyl}

\subsubsection{Weyl and Cotton tensors} A standard way of proving conformal flatness is to prove that a specific conformally-invariant component of the curvature tensor vanishes identically.

Let $W$ be $(3,1)$-Weyl tensor of $(M,\bar{g})$. It is conformally invariant, and when $\dim M \geq 4$, an open subset $U \subset M$ is conformally flat if and only if $W|_U = 0$ (see \cite{besse} Th. 1.159, 1.165). When $\dim M = 3$, the Weyl tensor is always zero. However, there is $(3,0)$-tensor $T$, called the Cotton tensor, which is also conformally invariant and such that an open set $U$ is conformally flat if and only if $T|_U = 0$.

\subsubsection{Proof of Corollary \ref{cor:local_flatness}}

Let $\|.\|$ denote a Riemannian metric on $M$. Let $X \in \a$ and assume that there exists a $\phi_X^t$-invariant, $\phi_X^t$-ergodic measure $\mu$ which admits a non-zero uniform vertical Lyapunov spectrum. Then, let $x \in M$, $g \in G$, $r>0$, $(\gamma_k)$ and $(T_k)$ be as in the conclusions of Proposition \ref{prop:continuous_to_discrete}. Applying Proposition \ref{prop:stability} to $U=B(x,r)$ and $(f_k)=(\gamma_k)$, we obtain that $(\gamma_k)$ is uniformly $(T_k)$-Lyapunov regular with exponent $-1$ at any point in $B(x,r)$, in particular $\frac{1}{T_k} \log \|D_y \gamma_k v\| \rightarrow -1$ for any $y \in B(x,r)$ and $v \in T_yM \setminus \{0\}$.

Let us assume first $\dim M \geq 4$. By compactness of $M$, there is $C>0$ such that for all $y \in M$ and $u,v,w \in T_yM$, $\|W_y(u,v,w)\| \leq C \|u\|\|v\| \|w\|$. Let now $y \in B(x,r)$ and $u,v,w \in T_yM$. The $\gamma_k$-invariance of $W$ means
\begin{equation*}
( \gamma_k )_* W_y(u,v,w) = W_{\gamma_k.y} ((\gamma_k)_*u,(\gamma_k)_*v,(\gamma_k)_*w)
\end{equation*}
If $W_y(u,v,w)$ was non-zero, then we would have $\frac{1}{T_k} \log \|( \gamma_k )_* W_y(u,v,w)\| \rightarrow -1$. But the conformal invariance of $W$ implies
\begin{equation*}
\|( \gamma_k )_* W_y(u,v,w)\| \leq  C \|(\gamma_k)_*u\|\|(\gamma_k)_*v\|\|(\gamma_k)_*w\|.
\end{equation*}
Thus, we would obtain $\bar{\lim} \frac{1}{T_k} \log \|( \gamma_k )_* W_y(u,v,w)\| \leq -3$, contradicting $W_y(u,v,w) \neq 0$. Thus, $W$ vanishes identically on $B(x,r)$. Since $[(g,x)] \in \Supp \mu$ by construction, this finishes the proof in the case $\dim M \geq 4$.

Assume now that $\dim M =3$, \textit{i.e.} $(M,\bar{g})$ is a closed Lorentzian $3$-manifold. 
The argument is essentially the same: the invariance of the Cotton tensor implies that for any $y \in B(x,r)$ and $u,v,w \in T_yM$,
\begin{equation*}
|T_y(u,v,w)| \leq C \|(\gamma_k)_*u\|\|(\gamma_k)_*v\|\|(\gamma_k)_*w\|,
\end{equation*}
for some $C >0$. Thus, we get $T_y(u,v,w)=0$ since the three factors in the right hand side converge to $0$, completing the proof of Corollary \ref{cor:local_flatness}.

\subsection{Conclusion}
\label{ss:conclusion}

We can now conclude that under the assumption $\Rk_{\R}G = p+1$, the whole manifold is conformally flat. It will follow from the

\begin{claim*}
\label{lem:flatness_in_K}
Any compact, $\Gamma$-invariant subset of $M$ intersects a conformally flat open set. 
\end{claim*}

\begin{proof}
Let $K$ be a compact $\Gamma$-invariant subset of $M$. Then, $G$ preserves the compact subset $K^{\alpha} := (G \times K)/\Gamma \subset M^{\alpha}$. Fix $A<G$ a Cartan subspace and let $B$ a Borel sugroup of $G$ containing $A$. By $B$-invariance of $K^{\alpha}$ and amenability of $B$, the same argument as in Section \ref{ss:bound_rank} gives the existence of a finite $A$-invariant, $A$-ergodic measure $\mu$ supported in $K^{\alpha}$, and whose projection on $G/\Gamma$ is the Haar measure.

By Lemma \ref{lem:unifom_element}, there exists $X \in \a$ whose vertical Lyapunov spectrum is reduced to $\{-1\}$. If $\Lambda \subset K^{\alpha}$ denotes the set of full measure $\mu$ where the conclusions of the higher-rank Oseledec's Theorem are valid, we choose $\mu'$ a probability measure $\phi_X^t$-invariant, $\phi_X^t$-ergodic such that $\mu'(\Lambda)=1$, so that the vertical Lyapunov spectrum of $X$ with respect to $\mu'$ is the same. By Corollary \ref{cor:local_flatness}, there exists $x \in K$ admitting a conformally flat neighborhood, because $\Supp \mu' \subset K^{\alpha}$ and $[(g,x)] \in K^{\alpha}$ implies $x \in K$.
\end{proof}

Let now $x \in M$, and consider the compact $\Gamma$-invariant subset $K = \bar{\Gamma.x}$. If $U$ is a conformally flat open set which meets $K$, then there is $\gamma \in \Gamma$ such that $\gamma.x \in U$, proving that $\gamma^{-1}U$ is a conformally flat neighborhood of $x$. Finally, any point of $M$ admits a conformally flat neighborhood, and Theorem \ref{thm:main} is established.

\section{Limit cases for exceptional root-systems}
\label{s:limit_cases}

In this section, we complete the proof of Theorem \ref{thm:exceptional}. We consider $\Gamma$ a uniform lattice in a simple Lie group $G$ whose restricted root-system $\Sigma$ is exceptional. We consider an unbounded conformal action $\alpha : \Gamma \rightarrow \Conf(M,\bar{g})$, where $(M,\bar{g})$ is a closed pseudo-Riemannian manifold of signature $(p,q)$, with $p \leq q$ and $p+q \geq 3$. We make the same hypothesis as in Section \ref{ss:optimal_index}: the index $p = \min(p,q)$ is assumed to be optimal for $\Gamma$, implying that $\g$ does not embed in $\so(p,q)$.

\subsection{General facts}

Let $\g = \k \oplus \p$ be a Cartan decomposition, $\a \subset \p$ a Cartan subspace, and $\m \subset \k$ the compact part of the centralizer of $\a$. Let $M_0<G$ the connected Lie group associated to $\m$. Until the end of this section, we fix $\mu$ an $AM_0$-invariant, $AM_0$-ergodic measure on $M^{\alpha}$ which projects to the Haar measure of $G/\Gamma$ and we note $H \subsetneq G$ the stabilizer of $\mu$, which is proper by Proposition \ref{prop:no_finite_invariant_measure}.

Since $\h$ contains $\a \oplus \m$, by Lemma 2.4 of \cite{BFH}, $\h$ is saturated by restricted root-spaces, \textit{i.e.} it is of the form $\h = \a \oplus \m \oplus \bigoplus_{\lambda \in S} \g_{\lambda}$ where $S \subset \Sigma$ is a subset. By Proposition 2.6 of \cite{BFH}, if $|\Sigma \setminus S| \leq r(\g)$, then $\h$ is parabolic. In particular, the inequality cannot be strict, because if not we would have $\h = \g$ by definition of $r(\g)$. Thus, $r(\g) \leq |\Sigma \setminus S|$.

As explained in Section 5.5 of \cite{BFH}, by $AM_0$-ergodicity and because $M_0$ centralizes $A$, there is a family $(\chi_1,\ldots,\chi_r)$ in $\a^*$ that coincides with the Lyapunov functionals of almost all $A$-ergodic component of $\mu$. We note that any $A$-ergodic component $\mu'$ of $\mu$ projects to the Haar measure of $G/\Gamma$ by $A$-ergodicity of the Haar measure.

Let $\Sigma_{\mu} \subset \Sigma$ be the set of restricted-roots that are not positively collinear to $\chi_i$ for $i=1,\ldots,r$. Then, $\Sigma_{\mu} \subset S$ because for almost all $A$-ergodic component $\mu'$ and for all $\lambda \in \Sigma_{\mu}$, $\lambda$ is not $\mu'$-resonnant, implying that $\mu'$ is $G_{\lambda}$-invariant and consequently $\mu$ is also $G_{\lambda}$-invariant by Proposition \ref{prop:resonance} (citation from \cite{BRHW}).

We observed in Section \ref{ss:optimal_index} that we always have $r(\g) \leq 2p+1$. We assume now that we are in the limit case $r(\g) = 2p+1$ - which does not occur for $\Sigma = E_6$ since $r(\mathfrak{e}_6)=16$. So, we implicitly assume $\Sigma \neq E_6$. By Proposition \ref{prop:lyapunov_exponents}, we get that $r \leq r(\g)$ and then $r(\g) \leq |\Sigma \setminus S| \leq |\Sigma \setminus \Sigma_{\mu}| \leq r \leq r(\g)$, showing that all these inequalities are equalities. It is worth noting that the inequality $|\Sigma \setminus \Sigma_{\mu}| \leq r$ follows from the fact that exceptional root-systems are reduced, so that two distinct restricted-roots not in $\Sigma_{\mu}$ are respectively positively collinear to two necessarily distinct Lyapunov functionals.

In particular, by Proposition 2.6 of \cite{BFH}, $\h$ is a parabolic subalgebra whose resonant codimension is minimal, equal to $r(\g)$. It implies that there exists a simple system of restricted roots $\Pi = \{\alpha_1,\ldots,\alpha_{\ell}\}$ and $j_0 \in \{1,\ldots,\ell\}$ such that
\begin{equation*}
S = (\{\sum_{j \neq j_0} n_j \alpha_j, \ n_j \in \Z_{\leq 0} \} \cap \Sigma) \cup \Sigma^+.
\end{equation*}
Moreover, $r = r(\g)$ and for all $1 \leq i \leq r$, there exists $\lambda_i \in \Sigma \setminus S$ such that $\lambda_i \in \R_{>0} \chi_i$, and the map $\{i \mapsto \lambda_i\}$ is injective. We note that $\lambda_i$ is uniquely determined because $\Sigma$ is reduced. We note $\ell_i^+ = \R_{>0}.\lambda_i$ the half-line generated by $\lambda_i$. 

We also note that $r=r(\g)=2p+1 \leq p+q = \dim M$ implies $p<q$.

Because $(\chi_1,\ldots,\chi_r)$ are the Lyapunov functionals of almost all $A$-ergodic components of $\mu$, up to reordering them, they satisfy the linear relations $\chi_1 + \chi_{2p+1} = \chi_2 + \chi_{2p} = \cdots = 2\chi_{p+1}$. This implies that for all $i \leq p$, $\ell_{p+1}^+$ is contained in the convex hull of $\ell_i^+$ and $\ell_{2p+2-i}^+$. Since $\{\lambda_1,\ldots,\lambda_{2p+1}\}$ is exactly the complement of $S$, we will see that in all cases, these conditions leave no choice for $\chi_{p+1}$ and that either no configuration exists, or only few possibilities can occur.

\subsection{Case $\Sigma = G_2$}
\label{ss:g2}

We have already seen that $k_{\Gamma} \geq 2$ when $\Sigma = \g_2$. As claimed in Remark \ref{rem:g2}, this inequality is sharp. Indeed, if $\g = \g_2^{(2)}$ is the real-split form of $\g_2$ and if $\rho : \g \rightarrow \gl(V)$ denotes the $7$-dimensional representation of $\g$, then $\rho(\g)$ is skew-symmetric with respect to a quadratic form of signature $(3,4)$. Consequently, $G_2^{(2)}$ acts locally faithfully and conformally on $\Ein^{2,3}$, implying that $k_{\Gamma} = 2$ for all cocompact lattice $\Gamma$ of $G_2^{(2)}$.

\subsection{Case $\Sigma = F_4$}
\label{ss:f4}

Let $\g$ be a Lie algebra whose restricted root-system is $F_4$. We assume here that $p=7$, ensuring that we are in the limit case $r(\g) = 2p+1=15$. We pick an orthonormal basis $(e_1,\ldots,e_4)$ of $\a^*$ such that $\Sigma$ is given by
\begin{itemize}
\item $\{\frac{1}{2}(\pm e_1 \pm e_2 \pm e_3 \pm e_4)\}$
\item $\{\pm e_i \pm e_j\}$, $1 \leq i,j \leq 4$
\item $\{\pm e_i\}$, $1 \leq i \leq 4$,
\end{itemize}
and such that $\alpha_1 = e_2-e_3$, $\alpha_2 = e_3-e_4$, $\alpha_3=e_4$, $\alpha_4 = \frac{1}{2}(e_1 -e_2-e_3-e_4)$. The fact that $\h$ has codimension $15$ implies that $j_0 = 1$ or $j_0=4$ (see Table 1 of \cite{BFH}).

\begin{enumerate}
\item \textbf{Case $j_0=4$}. In this situation, $\Sigma \setminus S$ is formed of the following $15$ restricted roots:
\begin{itemize}
\item $-e_1 \pm e_i$, $i \in \{2,3,4\}$
\item $-e_1$
\item $\frac{1}{2}(-e_1 \pm e_2 \pm e_3 \pm e_4)$.
\end{itemize}
Thus, the half-lines $\ell_1^+,\ldots,\ell_{15}^+$ are those generated by these vectors. The only vector $u$ in this list such that for all $v \neq u$ in this list, there is $w \notin \{u,v\}$ in this list such that $u$ can be written has a convex combination of $v$ and $w$ is $u=-e_1$. Indeed, for any other choice of $u$, the vector $v = -e_1$ would not have an associated vector $w$ such that $u \in \Span(v,w)$. This is because adding a multiple of $v$ to $w$ would not affect the components on $e_2,e_3,e_4$, and two different vectors in this list have distinct components on $e_2,e_3,e_4$.

 Consequently, there is $t > 0$ such that $\chi_8 = -t e_1$. It follows that up to an homothety and a permutation of $\{1,\ldots,15\}$ preserving the pairing $i \leftrightarrow 16-i$, we have 
\begin{itemize}
\item $\chi_8=-e_1$
\item For $i \in {1,2,3}$, $\chi_i=-e_1+e_{i+1}$ and $\chi_{16-i} = -e_1-e_{i+1}$
\item $\chi_4 = -e_1 + e_2+e_3+e_4$, $\chi_{12} = -e_1 - e_2-e_3-e_4$
\item $\chi_5 = -e_1 + e_2+e_3-e_4$, $\chi_{11} = -e_1 - e_2-e_3+e_4$
\item $\chi_6 = -e_1 + e_2-e_3+e_4$, $\chi_{10} = -e_1 + e_2-e_3-e_4$
\item $\chi_7 = -e_1 + e_2-e_3-e_4$, $\chi_{9} = -e_1 - e_2+e_3+e_4$
\end{itemize} 
Indeed, for instance, if $i$ is such that $\chi_i = s (-e_1+e_2)$, then the only choice for $\chi_{16-i}$ is to be of the form $\chi_{16-i} =s'(-e_1-e_2)$ and the condition $2\chi_8 = \chi_i + \chi_{16-i}$ implies $s=s'=t$. It works similarly for the other Lyapunov functionals.

\item \textbf{Case $j_0=1$}.  In this situation, $\Sigma \setminus S$ is formed of the following $15$ restricted roots:
\begin{itemize}
\item $-e_1$, $-e_2$
\item $-e_1-e_2$, $-e_1\pm e_3$, $-e_1 \pm e_4$, $-e_2 \pm e_3$, $-e_2 \pm e_4$
\item $\frac{1}{2}(-e_1-e_2 \pm e_3 \pm e_4)$.
\end{itemize}
These vectors generate the half-lines $\ell_1^+,\ldots,\ell_{15}^+$. Similarly to the previous case, we see that in this list, the only vector $u$ such that for any $v \neq u$, there exists $w \notin \{u,v\}$ such that $u$ is a convex combination of $v$ and $w$ is $u=-e_1-e_2$. We deduce similarly that up to an homothety and a suitable permutation, the Lyapunov functional are given by
\begin{itemize}
\item $\chi_8 = \frac{1}{2}( -e_1-e_2)$
\item $\chi_1 = -e_1$, $\chi_{15} = -e_2$
\item $\chi_2 = -e_1+e_3$, $\chi_{14} = -e_2-e_3$
\item $\chi_3 = -e_1-e_3$, $\chi_{13} = -e_2+e_3$
\item $\chi_4 = -e_1+e_4$, $\chi_{12} = -e_2-e_4$
\item $\chi_5 = -e_1-e_4$, $\chi_{11} = -e_2+e_4$
\item $\chi_6 = \frac{1}{2}(-e_1-e_2+e_3+e_4)$, $\chi_{10}=\frac{1}{2}(-e_1-e_2-e_3-e_4)$
\item $\chi_7 = \frac{1}{2}(-e_1-e_2+e_3-e_4)$, $\chi_9=\frac{1}{2}(-e_1-e_2-e_3+e_4)$
\end{itemize}
\end{enumerate}

We see that in both cases, there is a vector $v$ whose scalar product with all the $\chi_i's$ is constant. In the case $j_0=4$, the vector is $e_1$, and in the case $j_0=1$, the vector is $e_1+e_2$.

\textbf{Conclusion for $F_4$.} This proves that for any $AM_0$-invariant, $AM_0$-ergodic finite measure $\mu$ projecting to the Haar measure, and for almost every $A$-ergodic component $\mu'$ of $\mu$, with Lyapunov functionals $\chi_1,\ldots,\chi_{15}$, there is an element $X \in \a$ such that $\chi_1(X)=\cdots=\chi_{15}(X)=-1$. By Corollary \ref{cor:local_flatness} and Section \ref{ss:conclusion}, we get that $(M,\bar{g})$ is conformally flat.

\subsection{Case $\Sigma = E_8$} 
\label{ss:e8}

We assume here that the restricted root system of $\g$ is $E_8$. We assume $p=28$, ensuring that we are in the limit case $r(\g) = 2p+1 = 57$. We pick an orthonormal basis $(e_1,\ldots,e_8)$ of $\a^*$ such that $\Sigma$ is formed of the vectors
\begin{itemize}
\item $\pm e_i \pm e_j$, $1 \leq i < j \leq 8$
\item $\frac{1}{2}\sum_i (-1)^{n_i} e_i$, with $n_i \in \{0,1\}$, $\sum_i n_i$ even,
\end{itemize}
and such that $\alpha_8= \frac{1}{2}(e_8-e_7-\cdots-e_2+e_1)$, $\alpha_7=e_2+e_1$, $\alpha_6=e_2-e_1$,..., $\alpha_1=e_7-e_6$. Here, we necessarily have $j_0=1$ (see \cite{BFH}, Appendix A), and $\Sigma \setminus S$ is formed of the following $57$ restricted roots
\begin{itemize}
\item $\frac{1}{2}(-e_8-e_7 + \sum_{1\leq i\leq 6}(-1)^{n_i} e_i)$, $n_i \in \{0,1\}$, $\sum n_i$ even.
\item $-e_8-e_7$
\item $-e_8 \pm e_i$, $1 \leq i \leq 6$
\item $-e_7 \pm e_i$, $1 \leq i \leq 6$.
\end{itemize}

Similarly as before, we obtain that the only possibility for $\chi_{29}$ is to be a positive multiple of $-e_8-e_7$. Up to an homothety, we may assume $\chi_{29} = -\frac{1}{2}(e_8+e_7)$. Then, up to a suitable permutation, we get
\begin{itemize}
\item for $1\leq i \leq 6$, $\chi_i = -e_8 + e_i$, $\chi_{58-i} = -e_7 -e_i$, $\chi_{6+i} = -e_8 -e_i$, $\chi_{52-i} = -e_7 + e_i$.
\item $\{\chi_{13},\ldots,\chi_{45}\} \setminus \{\chi_{29}\} = \{\frac{1}{2}(-e_8-e_7 + \sum_{1\leq i\leq 6}(-1)^{n_i} e_i), \ n_i \in \{0,1\} \ \sum n_i \text{ even}\}$.
\end{itemize}
In the set $\{\chi_{13},\ldots,\chi_{45}\} \setminus \{\chi_{29}\}$, the association $\chi_i \leftrightarrow \chi_{58-i}$ is the obvious one, \textit{i.e.} we associate $\frac{1}{2}(-e_8-e_7 + \sum_{1\leq i\leq 6}(-1)^{n_i} e_i)$ with $\frac{1}{2}(-e_8-e_7 + \sum_{1\leq i\leq 6}(-1)^{n_i+1} e_i)$.

Thus, we see that the vector $e_8+e_7$ has a constant scalar product with all the $\chi_i$'s.

\textbf{Conclusion for $E_8$.} Similarly to the case of $F_4$, $(M,\bar{g})$ must be conformally flat.

\subsection{Case $\Sigma = E_7$}
\label{ss:e7}

In this situation, $r(\g)=27$ and we assume $p=13$. We embed $\a^*$ in $\R^8$ and pick an orthonormal basis $(e_1,\ldots,e_8)$ such that $\a^*$ is the orthogonal of $e_8+e_7$, and $\Sigma$ is formed of the following vectors
\begin{itemize}
\item $\pm(e_8-e_7)$
\item $\pm e_i \pm e_j$, $1 \leq i < j \leq 6$
\item $\pm \frac{1}{2}(e_8-e_7+\sum_{1 \leq i \leq 6} (-1)^{n_i} e_i)$, $n_i \in \{0,1\}$, $\sum n_i$ odd.
\end{itemize}
We may assume that $\alpha_7= \frac{1}{2}(e_8-e_7-\cdots-e_2+e_1)$, $\alpha_6=e_2+e_1$, $\alpha_5=e_2-e_1$,..., $\alpha_1=e_6-e_5$. Necessarily, $j_0 = 1$ and we obtain the following list for the vectors of $\Sigma \setminus S$ : 
\begin{itemize}
\item $\frac{1}{2}(-e_8+e_7-e_6+\sum_{1 \leq i \leq 5} (-1)^{n_i} e_i)$, $n_i \in \{0,1\}$, $\sum n_i$ even,
\item $-e_8+e_7$,
\item $-e_6\pm e_i$, $1 \leq i \leq 5$.
\end{itemize}
We claim that in this list, there is no vector $u$ such that for all $v \neq u$, there exists $w \notin \{u,v\}$ such that $u$ is a linear combination of $v$ and $w$. Indeed:
\begin{enumerate}
\item If $u = \frac{1}{2}(-e_8+e_7-e_6+\sum_{1 \leq i \leq 5} (-1)^{n_i} e_i)$, and if $v = -e_8+e_7$ then a vector $w$ such that $u \in \Span(v,w)$ must have non-zero components on $e_1,\ldots,e_6$. So $w$ must be of the form $\frac{1}{2}(-e_8+e_7-e_6+\sum_{1 \leq i \leq 5} (-1)^{n_i'} e_i)$. Since $v$ as no component on $e_1,\ldots,e_5$, we get $w=u$, which is absurd.
\item If $u=-e_8+e_7$, and if $v=-e_6+e_5$, then $w$ must have components on $e_8,e_7,e_6,e_5$, so it must be of the form $\frac{1}{2}(-e_8+e_7-e_6+\sum_{1 \leq i \leq 5} (-1)^{n_i} e_i)$, which is absurd.
\item If $u=-e_6\pm e_i$, and if $v = -e_8+e_7$, then, as above, $w$ must be of the form $\frac{1}{2}(-e_8+e_7-e_6+\sum_{1 \leq i \leq 5} (-1)^{n_i} e_i)$, which is again absurd.
\end{enumerate}

\textbf{Conclusion}: This contradicts the relations satisfied by the $\chi_i$'s and proves that when $\Sigma = E_7$ and $p=13$, there is no unbounded conformal action of $\Gamma$ on  $(M,\bar{g})$. We get as announced $k_{\Gamma} \geq 14$.

\bibliographystyle{alpha}
\bibliography{bibli}
\nocite{*}

\end{document}